\documentclass[a4paper,12pt,reqno]{amsart}
\usepackage{amsmath,graphics,amssymb}
\usepackage[dvipdfmx]{graphicx} \usepackage{bmpsize}
\usepackage{mathrsfs}
\newtheorem{theo}{Theorem}[section]
\newtheorem{lem}[theo]{Lemma}

\newtheorem{cor}[theo]{Corollary}

\theoremstyle{definition}

\theoremstyle{remark}
\newtheorem{rem}[theo]{Remark}

\numberwithin{equation}{section}

\newcommand{\M}{\operatorname{M}}

\newcommand{\de}{\operatorname{d}}

\newcommand{\Z}{\mathbb{Z}}
\newcommand{\R}{\mathbb{R}}

\newcommand{\al}{\alpha}
\newcommand{\be}{\beta}
\newcommand{\Ga}{\Gamma}
\newcommand{\ps}{\psi}

\newmuskip\pFqskip
\mathchardef\pFcomma=\mathcode`, 

\newcommand*\pFq[5]{%
  \begingroup
  \begingroup\lccode`~=`,
    \lowercase{\endgroup\def~}{\pFcomma\mkern\pFqskip}%
  \mathcode`,=\string"8000
  {}_{#1}F_{#2}\biggl[\genfrac..{0pt}{}{#3}{#4};#5\biggr]%
  \endgroup
}

\begin{document}

\title{Boundary dents, the arctic circle and the arctic ellipse}

\author{Mihai Ciucu and Christian Krattenthaler\\
  \\ \textnormal {with an appendix by} Michael Larsen}

\address{Indiana University, Department of Mathematics, Bloomington, Indiana 47405, USA}
\address{Universit\"at Wien, Fakult\"at f\"ur Mathematik, Oskar-Morgenstern-Platz 1, A-1090 Wien, Austria}

\thanks{M.C.'s research supported in part by Simons Foundation Collaboration Grant 710477.}

\begin{abstract} The original motivation for this paper goes back to the mid-1990's, when James Propp was interested in natural situations when the number of domino tilings of a region increases if some of its unit squares are deleted. Guided in part by the intuition one gets from earlier work on parallels between the number of tilings of a region with holes and the 2D Coulomb energy of the corresponding system of electric charges, we consider Aztec diamond regions with unit square defects along two adjacent sides. We show that for large regions, if these defects are at fixed distances from a corner, the ratio between the number of domino tilings of the Aztec diamond with defects and the number of tilings of the entire Aztec diamond approaches a Delannoy number.
  When the locations of the defects are not fixed but instead approach given points on the boundary of the scaling limit $S$ (a square) of the Aztec diamonds, we prove that, provided the line segment connecting these points is outside the circle inscribed in $S$, this ratio has the same asymptotics as the Delannoy number corresponding to the locations of the defects; if the segment crosses the circle, the asymptotics is radically different. We use this to deduce (under the assumption that an arctic curve exists) that the arctic curve for domino tilings of Aztec diamonds is the circle inscribed in $S$. We also discuss counterparts of this phenomenon for lozenge tilings of hexagons.
  \end{abstract}

\maketitle

\section{Introduction}

A domino tiling of a region $R$ on the square lattice is a covering of $R$ by dominos (unions of two unit squares that share an edge) so that there are no gaps or overlaps. The question of finding the number of different domino tilings of a given family of regions is one of great relevance in statistical physics --- it is one of the main goals in the dimer model. The case of an $m\times n$ rectangular region $R_{m,n}$ was solved in 1961 independently by Temperley and Fisher~\cite{TF} and by Kasteleyn~\cite{Kast}, who proved that\footnote{ For a region $R$ on the square lattice, we denote by $\M(R)$ the number of domino tilings of $R$.} 
\begin{equation}
\M(R_{m,n})=\prod_{j=1}^{\lceil m/2\rceil} \prod_{k=1}^{\lceil n/2\rceil}
\left(4\cos^2\frac{j\pi}{m+1}+4\cos^2\frac{k\pi}{n+1}\right).
\label{eaa}
\end{equation}

Another natural family of regions on the square lattice consists of the Aztec diamonds. The Aztec diamond of order $n$, denoted $AD_n$, is obtained by stacking $2n$ unit height strips of respective lengths $2,4,6,\dotsc,2n,2n,2n-2,\dotsc,4,2$ so that their centers line up along a common vertical lattice line (see Figure~\ref{fca} for a picture of $AD_8$). In their 1992 paper~\cite{EKLP}, Elkies, Kuperberg, Larsen and Propp proved that the number of domino tilings of $AD_n$ is given by the strikingly simple formula

\begin{equation}
\M(AD_n)=2^{n(n+1)/2}.
\label{eab}
\end{equation}

A natural refinement is to remove some fixed collection of unit squares from a given region~$R$, and consider the number of domino tilings of the leftover region~$R'$ . One can then ask: Is it possible to have
$M(R')>\M(R)$?
For the special case when $R$ is a square region this is believed not to be possible (this was mentioned by James Propp in a 1997 post on the domino email forum).
The current paper grew out of an exploration of this question for Aztec diamonds, and then considering the analogous question for lozenge tilings of hexagons.

Useful intuition can be derived from the first author's earlier work~\cite{sc,ov,ec,ef,gp,tf} on the behavior of the number $\M$ of perfect matchings of a large, fixed, finite subgraph $G$ of the square grid\footnote{Our earlier work also addresses the case of the hexagonal lattice and that of more general planar, two-periodic bipartite lattices.}
with a given collection of gaps, as the gaps move around (this applies just as well for gaps in perfect matchings of a subgraph of the hexagonal lattice, which correspond to gaps in lozenge tilings of regions on the triangular lattice; for brevity, in the following discussion we focus on the square lattice). Our results give strong support for the ``electrostatic conjecture'': as long as the gaps stay far from the boundary (i.e., in the bulk), $\log(\M)$ is given\footnote{ In the limit as $G$ becomes infinitely large, and the separations between the gaps approach infinity.}  by the 2D Coulomb energy of the system of charges obtained by regarding each gap as a point charge of magnitude equal to the number of white vertices minus the number of black vertices in the gap (in a proper coloring\footnote{ I.e., each edge is incident to one black and one white vertex.} of the square grid).

For general regions, the way $\log(\M)$ changes as the gaps move away from the bulk and interact with the boundary turns out to be more conveniently described by the steady-state heat energy of the system obtained by regarding each gap as a heat source or sink (of intensity given by the statistic that defined charge in the previous paragraph) in a uniform block of material having the shape of the region being tiled; see \cite{angle,rangle,tf}. However, in the case of Aztec diamonds (as well as that of hexagons on the triangular lattice), the interaction of the gaps with the boundary can still be understood to a good extent by the more suggestive electrostatic intuition.

Indeed, fix an Aztec diamond $AD_n$ and remove from it two unit squares, one white and one black\footnote{ We are using here the fact that domino tilings of a region~$R$ on the square lattice can be identified with perfect matchings of the planar dual graph of $R$.} (this is necessary in order for the leftover region to admit any domino tiling). Then, according to the electrostatic parallel described above, in the bulk these two gaps will have a tendency to attract (meaning that there are more tilings with them close by than with them far apart), reaching a maximum when they share an edge. However, this way the number of tilings of the region with the two gaps is just the number of tilings of $AD_n$ which contain the resulting domino, so it is not more than $\M(AD_n)$.

More successful in our quest to increase the number of tilings is to realize that the portion of the exterior of $AD_n$ which adjoins its southwestern side is in some sense a huge gap of positive charge\footnote{ We are assuming that the checkerboard coloring was chosen so that the unit squares of the Aztec diamond along its northwestern side are white.} (the same holds for the northeastern side). Therefore, away from the bulk there should be a great tendency of the black gap to be attracted to the southwestern or northeastern side, and similarly for the white gap to be attracted to the southeastern or northwestern side\footnote{ Given that the portions outside the boundary act like huge charges, the attraction of the gaps to the boundary should swamp the mild attraction tendency between the gaps.}. This suggests that a good location for the gaps (if we want to end up with a region that has more tilings than $AD_n$) is for instance to have the black gap glued to the southwestern boundary, and the white one glued to the southeastern boundary --- thus becoming dents instead of gaps.

Even for moderate size Aztec diamonds, with the dents close to the southern corner, we noticed that the ratio between the number of tilings of the dented region and the number of tilings of the entire Aztec diamond is close to being an integer. Then we discovered that the integers which they are close to are actually Delannoy numbers.

We prove this in Theorem~\ref{tba}.
The more general case when the Aztec diamond has $k$ dents along the southwestern side and $k$ along the southeastern one can be deduced from the case $k=1$ using a determinant identity published by Jacobi \cite[Eq.~(XX.4), p.\!\!\!~208]{Muir} (see Corollary~\ref{tcd}).

An interesting question is what happens in the $k=1$ case if the dents, instead of having fixed locations (as in Theorem~\ref{tba}), are positioned so that in the scaling limit their locations approach given points on the boundary of the square $S$, the scaling limit of Aztec diamonds. This led us to the surprising discovery that the asymptotic behavior of the ratio between the number of tilings of the dented and plain Aztec diamonds has two radically different regimes, depending on whether or not the line segment joining the dents crosses the circle inscribed in~$S$. This is stated in Theorem~\ref{tbaa}. Its implications for the arctic circle phenomenon --- including a new derivation (under the assumption that an arctic curve exists) of the fact that the arctic curve for domino tilings of Aztec diamonds is the inscribed circle --- are discussed after its statement, in Remarks~\ref{rem:2} and~\ref{rem:3}.
Sections~3 and~4 contain the proofs of our results, and also the explicit asymptotics of the ratio between the number of tilings of the dented Aztec diamond and the number of tilings of the plain Aztec diamond (see Theorem~\ref{tdb}). A surprising consequence of our explicit formulas is that in the case of $k$ dents on each of the bottom two sides of the Aztec diamond, provided all segments connecting dents on different sides cross the inscribed circle, the joint correlation of the dents is determined by the individual interactions of the dents with the corners
of the Aztec diamond; this is detailed in Theorem~\ref{tde},
Corollary~\ref{tdf} and Remark~\ref{rem:7}.

In Section~5 we consider the same questions for lozenge tilings of hexagons. We prove that the same phenomenon holds (with the ellipse inscribed in the hexagon playing the role of the circle inscribed in $S$), and find that the corresponding ratios approach binomial coefficients.

Our proofs of the results about domino tilings are essentially self-contained,
while those of the results about lozenge tilings are based on a counting formula
due to the first author and Fischer~\cite{boundarydents}.
They imply new versions of the arctic circle theorem for domino tilings of
the Aztec diamond
(see~\cite{JPS,CEP}) and of the analogous arctic curve result for lozenge tilings
of a hexagon (see~\cite{CLP}). Our new version of the arctic circle theorem
concerns Aztec diamonds with two dents on adjacent sides.
It states that, provided the line segment connecting the dents is outside
the inscribed circle, the arctic curve for domino tilings of such dented
Aztec diamonds is the union of a circle and a line segment (see
Theorem~\ref{tbc}). The analogous version for lozenge tilings of a similarly
dented hexagon says that the arctic curve is the union of an ellipse and
a line segment (see Remark~\ref{rem:8}). Both of these results need a
kind of ``folklore" fact from probability theory that says that lattice paths
from the origin (say) to a ``far away" point~$w$ are with high probability
close to the straight line segment connecting the origin and~$w$.
Since it seems that such a result has never been written down except in
special cases, the Appendix written by Michael Larsen provides a precise
statement and proof of that fact.

In view of the well-known interpretation of domino tilings (resp., lozenge tilings) as families of non-intersecting Delannoy paths, our results imply that, provided the line segment connecting the dents is outside the inscribed circle (resp., the inscribed ellipse), the path connecting the defects is asymptotically independent from the other paths in the family of non-intersecting lattice paths encoding the tiling.
For a study (from a different viewpoint) of dented hexagons in which the dents are not on adjacent sides but on alternating sides, see Condon's paper~\cite{Daniel}.

Our derivation of the fact that the arctic curve for domino tilings of Aztec diamonds is the inscribed circle (under the assumption that an arctic curve exists), and the analogous result for the ellipse inscribed in the hexagon, is very reminiscent of
the tangent method of Colomo and Sportiello~\cite{CS}. More precisely, these fit the framework of what Sportiello calls the ``2-refined tangent method'' in \cite[p.~33]{Sportiello}.
Our arguments are direct and completely self-contained.

After posting our paper on \texttt{arxiv.org}, the closely related work \cite{DR} by Debin and Ruelle (which we were not aware of when we wrote our paper) was pointed out to us by its second author. Indeed, in \cite{DR} a parallel analysis is carried out for the Aztec diamond (lozenge tilings of hexagons, the other example we study, are not considered in \cite{DR}). However, there is a crucial difference between the families of Delannoy paths considered by us and those in \cite{DR}: while the outermost path in our case connects two dents in the Aztec diamond, in \cite{DR} it connects two vertical segments on the sides of the $90^\circ$ wedge which contains the Aztec diamond --- and its starting and ending points are always {\it outside} the Aztec diamond itself.
Another difference is that our set-up allows us to deduce from our result the scaling limit of a family of non-intersecting Delannoy paths of the type resulting from domino tilings of Aztec diamonds, but having one starting point and one ending point removed (see Remark \ref{gaps}). The fact that this scaling limit is again determined by the line segment joining the removed starting and ending points seems to be a new and surprising result.
The two approaches thus complement each other, and the fact that they both lead to the same phenomenon is a satisfying confirmation of the 2-refined tangent method in this case. 

Two other points of contact with the previous literature are the work~\cite{BK} of Bufetov and Knizel and the paper~\cite{DFG} of Di Francesco and Guitter, in which the arctic curve for Aztec rectangles with dents along one side is determined.

\section{Main results}

For any $1\leq i,j\leq n$, let $AD_n^{i,j}$ be the region obtained from the Aztec diamond $AD_n$ by removing the $i$th unit square on its southwestern boundary and the $j$th unit square on its southeastern boundary (both counted from bottom to top); $AD_8^{2,4}$ is shown in Figure~\ref{fba}.

After some experimentation with the {\tt vaxmacs} program%
\footnote{This program, implemented by David B. Wilson, allows one to compute (among other things) the number of domino tilings of regions on the square lattice.
}
one quickly finds that for small values of $i$ and $j$, already for moderately large $n$ the ratio $\M(AD_n^{i,j})/\M(AD_n)$ is quite close to an integer. For instance, $\M(AD_{20}^{2,3})/\M(AD_{20})\approx 4.996$. On the other hand, 5 is the Delannoy number $D(1,2)$, where 
$D(k,l)$ is defined to be the number of paths on $\Z^2$ from $(0,0)$ to $(k,l)$ using only steps $(1,0)$, $(0,1)$ or $(1,1)$.

This turns out to hold in general.

\begin{theo}
\label{tba}
For any $1\leq i,j\leq n$, we have:

$(${\rm a}$)$. 
\begin{equation}
\frac{\M(AD_n^{i,j})}{\M(AD_n)}=\sum_{k=0}^{n-1}{\binom k {i-1}}{\binom k {j-1}}\frac{1}{2^{k+1}}.
\label{eba}
\end{equation}

$(${\rm b}$)$. 
\begin{equation}
\lim_{n\to\infty}\frac{\M(AD_n^{i,j})}{\M(AD_n)}=D(i-1,j-1),
\label{ebb}  
\end{equation}
where $D(k,l)$ is the Delannoy number, defined to be the number of paths on $\Z^2$ from $(0,0)$ to $(k,l)$ using only steps $(1,0)$, $(0,1)$ or $(1,1)$.


\end{theo}

\begin{rem} \label{rem:1}
We note that another expression for $\M(AD_n^{i,j})$ (equivalent to the one given in part (a) above) was proved by Saikia in \cite[Proposition~4.9]{Saikia}. The approach in~\cite{Saikia} is to first
come up with the expression,
    and then prove  it by induction, using the graphical condensation method of Kuo~\cite{Kuo}. By contrast, we prove our formula directly, using a bafflingly simple factorization of the Delannoy matrix that ought to be better known (this is presented at the beginning of Section~3).
\end{rem}
    
\medskip
By Theorem~\ref{tba}(b),
$\M(AD_n^{i,j})$ and  $\M(AD_n)\cdot D(i-1,j-1)$ have the same $n\to\infty$ asymptotics for fixed $i$ and $j$. It is natural to ask what happens if $i$ and $j$ are allowed to grow with $n$. The answer is given by the following surprising result.

\begin{theo}
\label{tbaa}
Let $C$ be the circle inscribed in the unit square.
Then as $n,i,j\to\infty$ so that $i/n\to a$ and $j/n\to b$, where $0<a,b<1$, we have 
\begin{equation}
\lim_{n\to\infty}\frac{\M(AD_n^{i,j})}{\M(AD_n) D(i-1,j-1)}=
\begin{cases}
  1,  & \text{\rm if the segment $[(a,0),(0,b)]$ is outside $C$}, \\
  0,  & \text{\rm if the segment $[(a,0),(0,b)]$ crosses $C$}.
\end{cases}
\label{ebbbb}  
\end{equation}

\end{theo}

\begin{figure}[t]
\vskip0.3in
\centerline{
\hfill
{\includegraphics[width=0.15\textwidth]{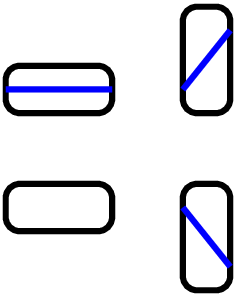}}
\hfill
}
\caption{\label{fbaa} Four types of marked dominos.}
\vskip0.3in
\centerline{
\hfill
{\includegraphics[width=0.43\textwidth]{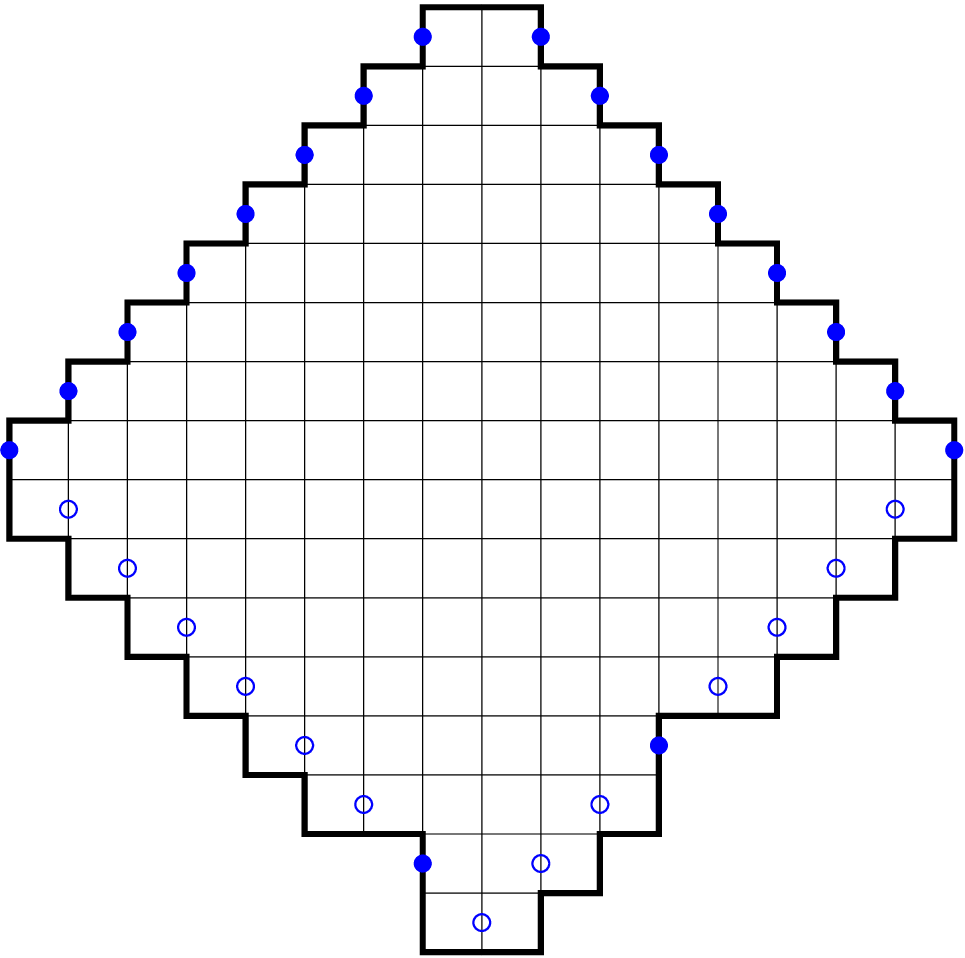}}
\hfill
{\includegraphics[width=0.43\textwidth]{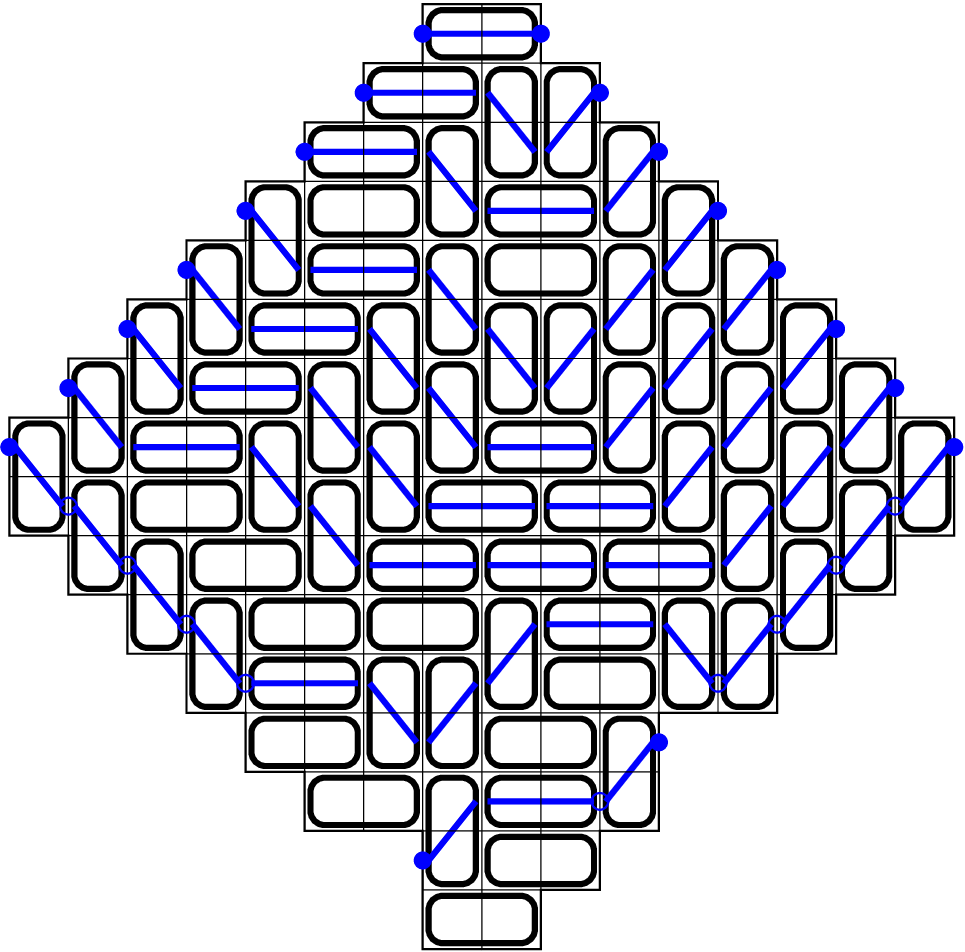}}
\hfill
}
\caption{\label{fba} The dented Aztec diamond $AD_{n}^{i,j}$ for $n=8$, $i=2$, $j=4$: Marked points on the boundary (left); the family of non-intersecting Delannoy paths that connect the marked points corresponding to a domino tiling (right).}
\end{figure}

\begin{rem} \label{rem:2}
Domino tilings of a simply connected region~$R$ on the square lattice are well-known to be in bijection with families of non-intersecting Delannoy paths having certain fixed starting and ending points on the boundary of~$R$ (see~\cite{LRS}). This bijection works as follows. Mark the midpoint $v$ of a vertical unit lattice segment on the boundary of $R$, and mark also the midpoints of all vertical unit lattice segments that can be obtained from $v$ by translating it by the vector $(x,y)$, where $x,y\in\Z$ and $x+y$ is even\footnote {Note that depending on the choice of $v$, there are two different possible $(\sqrt{2}\Z)\times(\sqrt{2}\Z)$ lattices that can result in this way; choosing one or the other can make a crucial difference in a given situation.}. Consider a domino tiling $T$ of the region~$R$. Then there is a unique way of marking the dominos in one of the four ways shown in Figure~\ref{fbaa} such that the marked points on their boundaries agree among each other and with the marked points on the boundary of $R$. This creates a family of non-intersecting Delannoy paths connecting the marked points on the boundary of $R$. For one domino tiling of the dented Aztec diamond $AD_8^{2,4}$ this correspondence is illustrated in Figure~\ref{fba}.

For the region $AD_n^{i,j}$, this implies that its tilings are in bijection with families of $n+1$ non-intersecting Delannoy paths, the first of which connects the two dents, and the last $n$ of which encode a domino tiling of the full Aztec diamond $AD_n$.

By Theorem~\ref{tbaa}, as long as the segment $[(a,0),(0,b)]$ is outside the circle $C$ inscribed in the unit square, the number of such families of $n+1$ non-intersecting paths is asymptotically the same as the number of {\it independent} pairs $(P,\mathscr P)$, where $P$ is a Delannoy path connecting the two dents, and $\mathscr P$ is a family of $n$ non-intersecting Delannoy paths encoding a tiling of the full Aztec diamond $AD_n$. Therefore, in the limit as $n\to\infty$, sampling uniformly at random from the former set is the same as sampling uniformly at random from the latter set --- which amounts to a pair consisting of a Delannoy path $P$ connecting the two dents chosen uniformly at random, and an $n$-tuple $\mathscr P$ of Delannoy paths corresponding to a domino tiling of $AD_n$ chosen uniformly at random (and independently of $P$).

\begin{figure}[t]
\vskip0.2in
\centerline{
\hfill
{\includegraphics[width=0.65\textwidth]{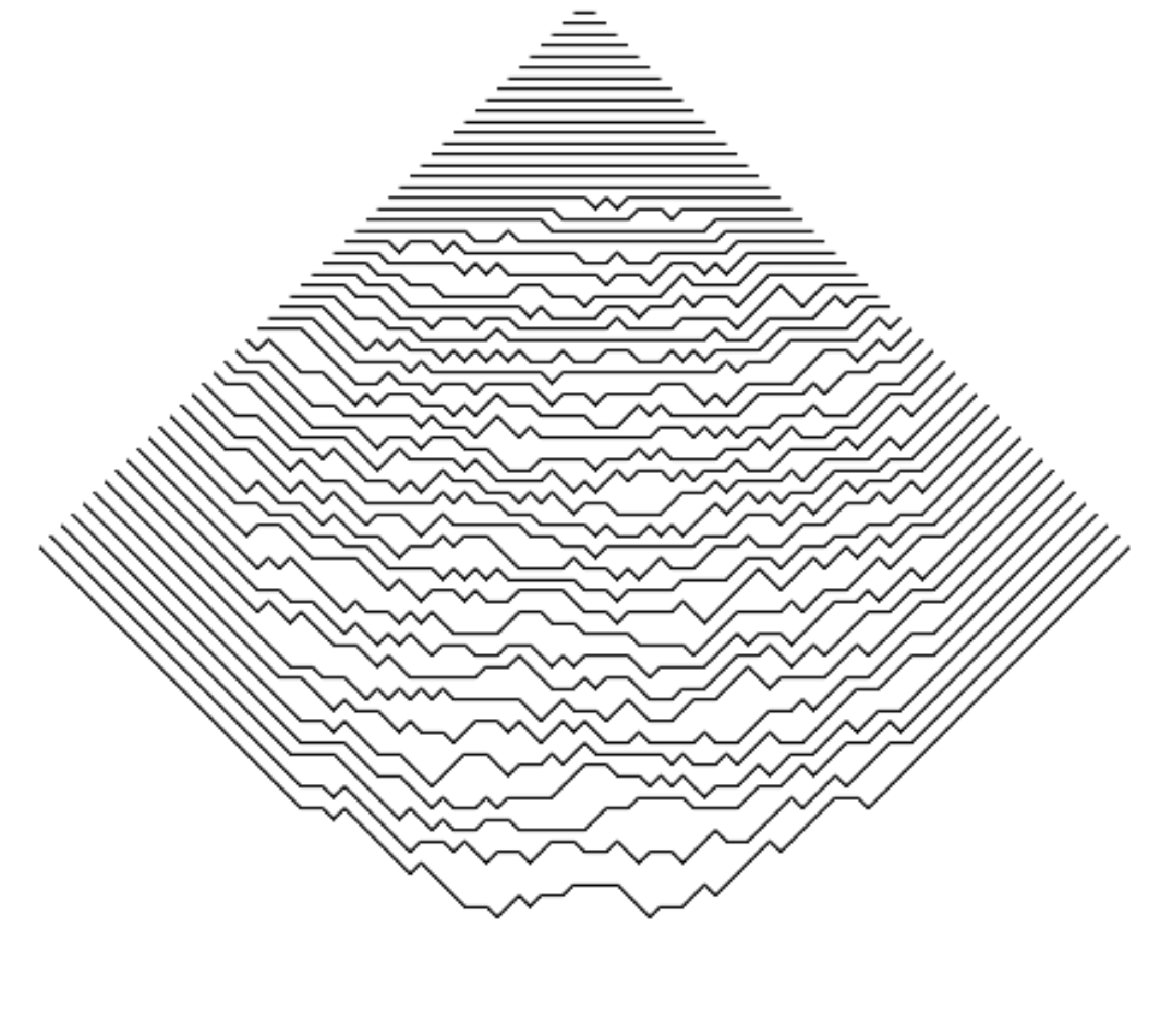}}
\hfill
}
\caption{\label{fbb} The family of non-intersecting Delannoy paths corresponding to a typical domino tiling of $AD_{50}$.}
\end{figure}

The typical form of the $n$-tuple $\mathscr P$ of paths follows by the arctic circle theorem~\cite{JPS,CEP}; see Figure~\ref{fbb} (this figure was produced using Antoine Doeraene's Aztec Diamond Generator found at {\tt https://sites.uclouvain.be/aztecdiamond/}; see also~\cite{JRV}). The typical Delannoy path $P$ connecting the two dents approaches in the limit the segment\footnote{ See $(ii)$ in Remark~\ref{rem:3} for a precise statement. Its proof is provided in the Appendix, which is due to Michael Larsen.
} $[(a,0),(0,b)]$.
Therefore, a corollary of Theorem~\ref{tbaa} is that, provided $[(a,0),(0,b)]$ is outside the circle $C$, in the scaling limit the family of paths encoding a typical domino tiling of $AD_n^{i,j}$ (where $i/n\to a$, $j/n\to b$) looks like the family of paths encoding a typical domino tiling of the undented Aztec diamond, with an additional path that in the limit becomes the line segment $[(a,0),(0,b)]$. This proves then the following new variant of the arctic circle phenomenon.
\end{rem}

\begin{theo}
\label{tbc}
Consider the scaling limit of the dented Aztec diamonds $AD_n^{i,j}$ as $n,i,j\to\infty$ so that $i/n\to$ and $j/n\to b$, and let $S$ denote the scaling limit of the boundary. Then, provided the line segment $s=[(a,0),(0,b)]$ is outside the circle $C$ inscribed in $S$, the arctic curve for domino tilings of the dented Aztec diamonds $AD_n^{i,j}$ is\footnote{ I.e., with probability approaching 1 as $n\to\infty$, around each point $u\in S\setminus(\overline{int(C)}\cup s)$, all the dominos in the tiling have the same type as the domino that fills the corner of the Aztec diamond which is closest to $u$, and $\overline{int(C)}\cup s$ is the maximal set with this property (two dominos are said to have the same type if one is obtained from the other by a translation which preserves the checkerboard coloring of the square lattice).} $C\cup s$.
\end{theo}

\begin{rem} \label{gaps}
  Tilings of the dented Aztec diamond $AD_n^{i,j}$ can also be encoded by non-intersecting Delannoy paths in a different way. Namely, instead of the points marked in Figure \ref{fba}, mark the points obtained from them by translating them one unit down. Note that after this translation, only the bottom four of the original $2n+2$ starting and ending points remain on the boundary of $AD_n^{i,j}$ (the others are moved to the interior), while $2n-6$ of the marked points that were originally in the interior end up on the lower half of the boundary. These $2n-2$ points are the starting and ending points of a new familly of $n-1$ non-intersecting Delannoy paths, which also encodes the tiling (this new family is obtained as before, by placing on the dominos the markings of Figure \ref{fbaa} so that the endpoints of the markings agree with each other and with the above-described $2n-2$ marked points on the boundary).

This new family of non-intersecting Delannoy paths is almost identical to one corresponding to the plain Aztec diamond $AD_n$ --- the only difference is that the $i$th starting point and the~$j$th ending point are removed. It is an interesting (and seemingly hard) question to ask: What will be the effect of these two gaps in the starting and ending points on the typical shape of such a family of non-intersecting paths?

 Since these families of paths are in one to one correspondence with tilings of $AD_n^{i,j}$, which in turn are in one to one correspondence with the families of non-intersecting lattice paths considered in Remark \ref{rem:2}, the answer to this question follows by Theorem \ref{tbc}: the typical paths look like what one obtains by reflecting Figure \ref{fbb} across the horizontal symmetry axis of the Aztec diamond, with one change in the bottom frozen region, which is illustrated in Figure~\ref{fgap}. Namely, provided that the line segment $L$ joining the two gaps is outside the arctic circle $C$, in the neighborhood of $L$ the Delannoy paths will also have some diagonal steps, which take them across the ``rift'' created in the frozen region by the original Delannoy path connecting the boundary dents (shown in blue in Figure \ref{fba}).

\end{rem}

\begin{figure}[t]
\vskip0.2in
\centerline{
\hfill
{\includegraphics[width=0.65\textwidth]{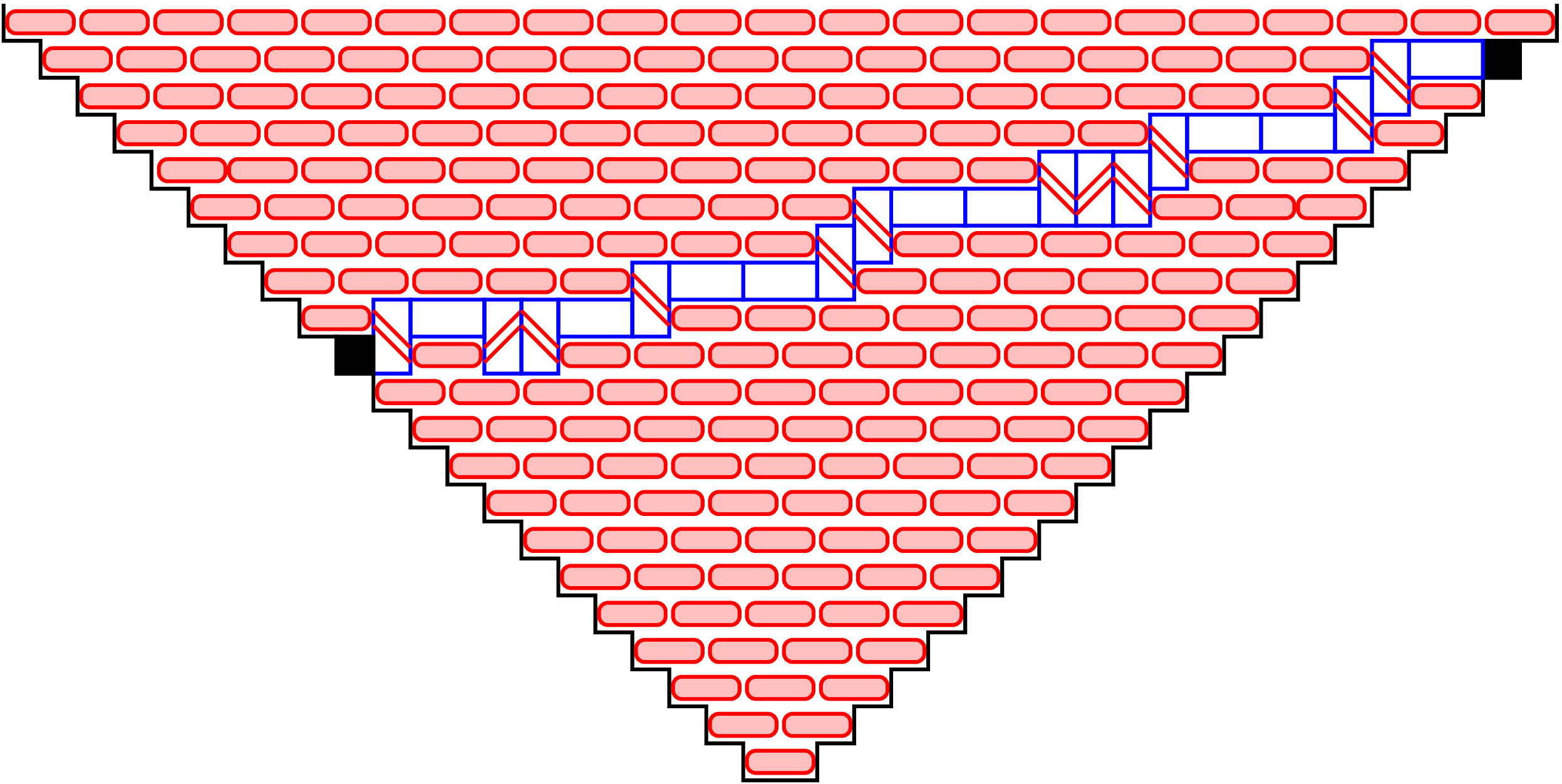}}
\hfill
}
\caption{\label{fgap} The way two unit boundary dents change the appearance of a typical family of non-intersecting Delannoy paths in the frozen region (assuming the line segment $L$ joining the gaps is outside the arctic circle). The paths are obtained following the horizontal portions consisting of the shaded horizontal dominos and the indicated diagonal steps connecting them, which make them pass across the ``rift'' produced by the single blue Delannoy path (of the original encoding) in this frozen region. Compared to the top of Figure~\ref{fbb}, where the paths consist completely of horizontal steps, in this case there are some diagonal steps in the neighborhood of $L$, which are needed (because of the effect of the gaps) in order to match the starting points to the ending points.}
\end{figure}

\begin{rem} \label{rem:3}
One can deduce from Theorem~\ref{tbaa} that measuring the ``observable'' $\M(AD_n^{i,j})$ reveals that the arctic curve for domino tilings of Aztec diamonds is the inscribed circle. 

In order to make this argument, we need to assume that an arctic curve for domino tilings of Aztec diamonds exists. More precisely, assume that there exists a convex arc $\mathscr C$ tangent to the two bottom sides of
the square $S$
so that if $\widehat{\mathscr C}$ is the closed region bounded below by the arc ${\mathscr C}$ and above by the boundary of $S$ we have:

\medskip
$(i)$ for any open set $U$ containing $\widehat{\mathscr C}$, the family of $n$ non-intersecting Delannoy paths corresponding to a tiling of $AD_n$ is contained in $U$ with probability approaching 1, as $n\to\infty$; on the other hand, for any open set $V\subsetneq \widehat{\mathscr C}$, with probability approaching 1 this family of $n$ non-intersecting Delannoy paths is {\it not} contained in $V$.

\medskip
We claim that under this assumption it follows from Theorem~\ref{tbaa} that the arctic curve --- which, by symmetry, is the union of $\mathscr C$ with its rotation by $90^\circ$,  $180^\circ$ and  $270^\circ$ around the center of $S$ --- is the inscribed circle $C$.

We will also use in our arguments the following fact, which is proved in the Appendix:

\medskip
$(ii)$ let $0<a,b<1$; then for any open set $U$ containing the segment $[(a,0), (0,b)]$, the image through a homothethy of factor $1/n$ through the origin of a Delannoy path chosen uniformly at random from the set of paths connecting the lattice points $(i,0)$ and $(0,j)$ is contained in $U$ with probability approaching 1, as $n,i,j\to\infty$ so that $i/n\to a$, $j/n\to b$.

\bigskip
Indeed, suppose that the segment $[(a,0), (0,b)]$ does not cross $\mathscr C$. Then there exist disjoint open sets $U$ and $V$ so that $[(a,0), (0,b)]\subset U$ and $\widehat{\mathscr C}\subset V$. By $(ii)$, basically all Delannoy paths $P$ connecting the two dents are contained in $U$, while by $(i)$, basically all $n$-tuples $\mathscr P$ of non-intersecting Delannoy paths connecting the remaining $n$ pairs of starting and ending points on the boundary of $AD_n^{i,j}$ (equivalently, connecting the same pairs of points on the boundary of $AD_n$) are contained in $V$. Therefore, basically all such pairs $(P,\mathscr P)$ (with $P$ and $\mathscr P$ chosen independently) form a family of $n+1$ non-intersecting Delannoy paths, thus encoding a tiling of $AD_n^{i,j}$. This implies that 
\begin{equation}
\lim_{n\to\infty}\frac{\M(AD_n^{i,j})}{\M(AD_n) D(i-1,j-1)}=1
\label{ebc}
\end{equation}
(because there are $D(i-1,j-1)$ paths connecting the two dents).

By Theorem~\ref{tbaa},
it follows that the segment $[(a,0), (0,b)]$ does not cross the inscribed circle $C$. Thus, the fact that the segment $[(a,0), (0,b)]$ does not cross $\mathscr C$, implies that the same segment does not cross the inscribed circle $C$ either. Since $\mathscr C$ is convex, this implies that
$\widehat C\subset\widehat{\mathscr C}$, where $\widehat C$ is the closed region bounded below by the lower quarter of $C$ and bounded above by the boundary of $S$.

Suppose now that the segment $[(a,0), (0,b)]$ crosses the arc $\mathscr C$. Then by $(i)$ and $(ii)$, for all but a negligible fraction of the pairs $(P,\mathscr P)$, the family of $n+1$ paths it forms fails to be non-intersecting. It follows that in this case
\begin{equation}
\lim_{n\to\infty}\frac{\M(AD_n^{i,j})}{\M(AD_n) D(i,j)}=0.
\label{ebd}
\end{equation}
But then by Theorem~\ref{tbaa} it follows that the segment $[(a,0), (0,b)]$ crosses the inscribed circle~$C$. Therefore, the fact that the segment $[(a,0), (0,b)]$ crosses $\mathscr C$, implies that the same segment also crosses the inscribed circle $C$. This in turn implies that $\widehat{\mathscr C}\subset\widehat{C}$. Thus the arc $\mathscr C$ must be the lower quarter of $C$, and by symmetry the arctic curve must be the inscribed circle $C$, as claimed.

\medskip
The explicit formulas for the asymptotics of $\M(AD_n^{i,j})/\M(AD_n)$ in the two regimes are given in Theorem~\ref{tdb}. An immediate consequence is the formula for the limiting ``helmet'' surface for
$\frac {1} {n}\log(\M(AD_n^{i,j})/\M(AD_n))$
as a function of $a$ and $b$ (see Corollary~\ref{cor:tdc} and Figure~\ref{fda}). A surprising consequence of Theorem~\ref{tdb} is that if the dents in $AD_n^{i_1,\dotsc,i_k,j_1,\dotsc,j_k}$ are positioned so that each segment connecting an $i$-dent with a $j$-dent crosses the inscribed circle, then the asymptotics of $\M(AD_n^{i_1,\dotsc,i_k,j_1,\dotsc,j_k})/\M(AD_n)$ is given by an explicit product of linear factors; this is presented in Theorem~\ref{tde} and Corollary~\ref{tdf}, and interpreted in Remark~\ref{rem:7}.

We present analogs of Theorem~\ref{tbaa} and Theorem~\ref{tbc} for lozenge tilings of dented hexagons (when the dents are on two adjacent sides) in Section~5 (see Theorem~\ref{teb} and Remark~\ref{rem:8}). It is noteworthy that the phenomenon we noticed for dented Aztec diamonds holds here as well: when each segment connecting two dents on different sides crosses the ellipse inscribed in the hexagon, the asymptotics of the ratio between the number of lozenge tilings of the dented and undented hexagon is given by an explicit product of linear factors (see Theorem~\ref{teg} and a discussion of its geometric interpretation in Remarks~\ref{rem:10} and~\ref{rem:12}).
\end{rem}

\section{Proof of Theorem \ref{tba}}

Let the $n\times n$ matrices $L_n$ and $S_n$ be defined by
\begin{equation}
  L_n=\left({\binom i  j}\right)_{0\leq i,j<n},\ \ \
  S_n=\left({\binom {i+j} j}\right)_{0\leq i,j<n}.
\label{ecc}
\end{equation}
Then it is well known (see e.g.~\cite{ES}) that
\begin{equation}
  S_n=L_n\,L_n^{\top}.
 \label{ecd}
\end{equation}
It turns out that a simple modification of the right-hand side above produces the Delannoy matrix
\begin{equation}
  D_n=\left(D(i,j)\right)_{0\leq i,j<n}.
 \label{ebe}
\end{equation}

\begin{theo} 
\label{tca}
We have
\begin{equation}
D_n=L_n
\begin{bmatrix}
  1 &   &    &       & \\
    & 2 &    &       & \\
    &   & 4  &       & \\
    &   &    &\ddots & \\
    &   &    &       & 2^{n-1}
\end{bmatrix}
L_n^{\top}.
\label{ebf}
\end{equation}

\end{theo}

This beautiful factorization seems to have been noticed first by Douglas Zare, who stated it in a message on the domino email forum in 1999.

\begin{proof} Let $i,j\in\{0,1,\dotsc,n-1\}$. Then the $(i,j)$-entry of the product of the three matrices on the right-hand side of~\eqref{ebf} is
\begin{align}
\sum_{k=0}^{n-1}L_n(i,k)\, 2^k L_n^\top(k,j)
&=\sum_{k=0}^{n-1} 2^k{\binom i k} {\binom j k}\\
&=\sum_{k=0}^{\min(i,j)} 2^k{\binom i k} {\binom j k},
\label{ebg}
\end{align}
which is a classical expression for the Delannoy number $D(i,j)$ (see e.g.~\cite{BS}). \end{proof}

As pointed out by Zare in his 1999 email on the domino forum, an immediate corollary of this factorization is a very short proof of the Aztec diamond theorem of Elkies, Kuperberg, Larsen and Propp~\cite{EKLP}.

\begin{cor}
\label{tcb}
The number of domino tilings of the Aztec diamond of order $n$ is $2^{n(n+1)/2}$.
\end{cor}

\begin{proof} 
Mark the $2n$ midpoints of the vertical unit segments on the boundary of the lower half of the Aztec diamond $AD_n$. Let $u_1,\dotsc,u_n$ be the leftmost $n$ marked points, and $v_1,\dotsc,v_n$ the remaining $n$ marked points (both listed from bottom to top). Mark also the midpoint $w$ of the vertical segment just below $AD_n$ along its vertical symmetry axis (see Figure~\ref{fca}), and let $u_0=v_0=w$.

\begin{figure}[t]
\vskip0.2in
\centerline{
\hfill
{\includegraphics[width=0.45\textwidth]{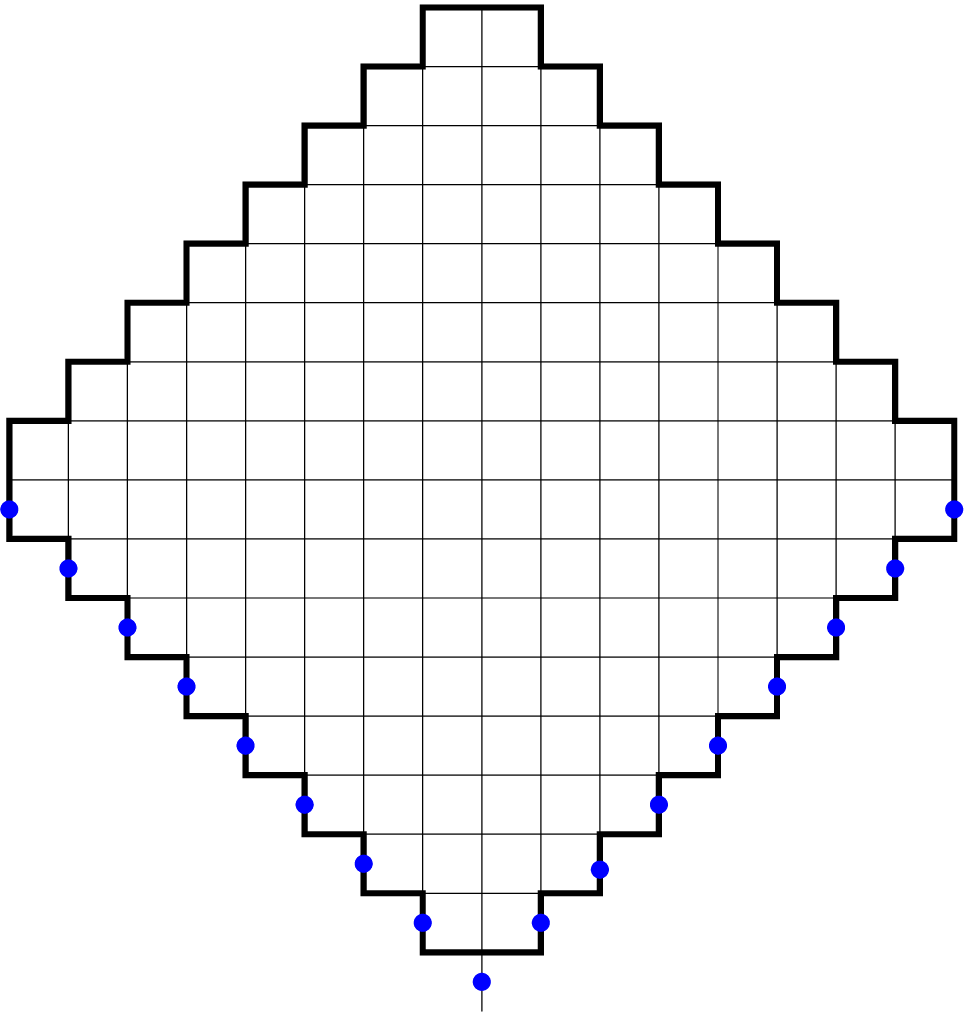}}
\hfill
}
\caption{\label{fca} Marked points on the boundary of $AD_n$ for $n=8$. The marked points on the left, from bottom to top, are $u_0,u_1,\dotsc,u_n$, while those on the right are $v_0,v_1,\dotsc,v_n$ ($u_0=v_0$ is the point at the very bottom).}
\end{figure}

The set of midpoints of vertical unit segments on the square grid with the unit square to the left of them having the same color as the unit squares to the left of these marked points (in a checkerboard coloring) form a lattice ${\mathcal L}$ isomorphic to $\Z^2$; choose the positive directions of the coordinate axes in ${\mathcal L}$ to point southeast and northeast. For $u,v\in{\mathcal L}$, denote by $D(u,v)$ the number of Delannoy paths starting at $u$ and ending at $v$. Then the matrix $(D(u_i,v_j))_{0\leq i,j\leq n}$ is precisely the Delannoy matrix $D_{n+1}$. Therefore, by the Lindstr\"om--Gessel--Viennot theorem (see~\cite{L,GV}), the number of families of non-intersecting Delannoy paths with starting points $u_0,\dotsc,u_n$ and ending points $v_0,\dotsc,v_n$ is equal to $\det D_{n+1}$. 

Note that as a consequence of the geometric positions of these starting and ending points, all Delannoy paths in such a non-intersecting family must in fact be contained in $AD_n$ (the inclusion of the points $u_0$ and $v_0$ was necessary for this to hold; for families in which the paths are allowed to intersect this is not necessarily the case). Therefore, by the bijection described in Remark~\ref{rem:2}, the families of non-intersecting Delannoy paths with starting points $u_0,\dotsc,u_n$ and ending points $v_0,\dotsc,v_n$ can be identified with domino tilings of $AD_n$. This implies $\M(AD_n)=\det D_{n+1}$. Since by Theorem~\ref{tca}
we have $\det(D_{n+1})=2^{1+2+\dotsc+n}$, the proof is complete. \end{proof}

We note that another short proof, using Schr\"oder paths, was given by Eu and Fu~\cite{EF}.

The factorization~\eqref{ebf} can also be used to obtain the following result.

\begin{lem}
\label{tcc} $(${\rm a}$)$. The $(i,j)$ entry of the inverse of the Delannoy matrix $D_n$ is given by
\begin{equation}
D_n^{-1}(i,j)=2(-1)^{i+j}\sum_{k=0}^{n-1}{\binom k i}{\binom k j}\frac{1}{2^{k+1}}.
\label{ebh}  
\end{equation}

$(${\rm b}$)$. In the limit as $n\to\infty$, each fixed position entry of the inverse of $D_n$ approaches, up to sign, the double of the corresponding entry of $D_n$:
\begin{equation}
\lim_{n\to\infty}D_n^{-1}(i,j)=2(-1)^{i+j}D(i,j).
\label{ebi}  
\end{equation}

\end{lem}  

\begin{proof} It is readily checked that the matrix $L_n$ defined by~\eqref{ebc} has inverse
\begin{equation}
L_n^{-1}=\left((-1)^{i+j}{\binom i  j}\right)_{0\leq i,j<n}.
\label{ebj}
\end{equation}
Indeed, denote by $U_n$ the matrix on the right-hand side of~\eqref{ebj}. Then if $i\geq j$, the $(i,j)$-entry of $L_n\,U_n$ is
\begin{align}
\sum_{k=0}^{n-1}L_n(i,k)\,U_n(k,j)&=\sum_{k=0}^{n-1}{\binom i k}\cdot(-1)^{k+j}{\binom k j}\\
&=\sum_{k=0}^{n-1}(-1)^{k+j}{\binom {i-j} {i-k}}{\binom i j}
&=(-1)^j{\binom i j}\sum_{k=0}^{n-1}(-1)^{k}{\binom {i-j} {i-k}},
\label{ebk}
\end{align}
which is zero for $i>j$ by the binomial theorem, and 1 for $i=j$. A similar calculation verifies that the $(i,j)$-entry of $L_n\,U_n$ is zero if $i<j$. This proves~\eqref{ebj}.

Taking inverses in the factorization~\eqref{ebf} yields
\begin{equation}
D_n^{-1}=(L_n^{-1})^{\top}
\begin{bmatrix}
  1 &   &    &       & \\
    & 1/2 &    &       & \\
    &   & 1/4  &       & \\
    &   &    &\ddots & \\
    &   &    &       & 1/2^{n-1}
\end{bmatrix}
L_n^{-1}.
\label{ebl}
\end{equation}
Therefore, by \eqref{ebj}, the $(i,j)$-entry of $D_n^{-1}$ is
\begin{equation}
\sum_{k=0}^{n-1}(-1)^{i+k}{\binom k i}\frac{1}{2^k}(-1)^{k+j}{\binom k j}
=2(-1)^{i+j}\sum_{k=0}^{n-1}{\binom k i}{\binom k j}\frac{1}{2^{k+1}},
\label{ebm}
\end{equation}
which proves part (a). Part (b) follows from another classical expression for the Delannoy numbers:
\begin{equation}
D(i,j)=\sum_{k=0}^{\infty}{\binom k i}{\binom k j}\frac{1}{2^{k+1}}
\label{ebn}
\end{equation}
(see e.g.~\cite{BS}). \end{proof}


\begin{figure}[t]
\vskip0.3in
\centerline{
\hfill
{\includegraphics[width=0.43\textwidth]{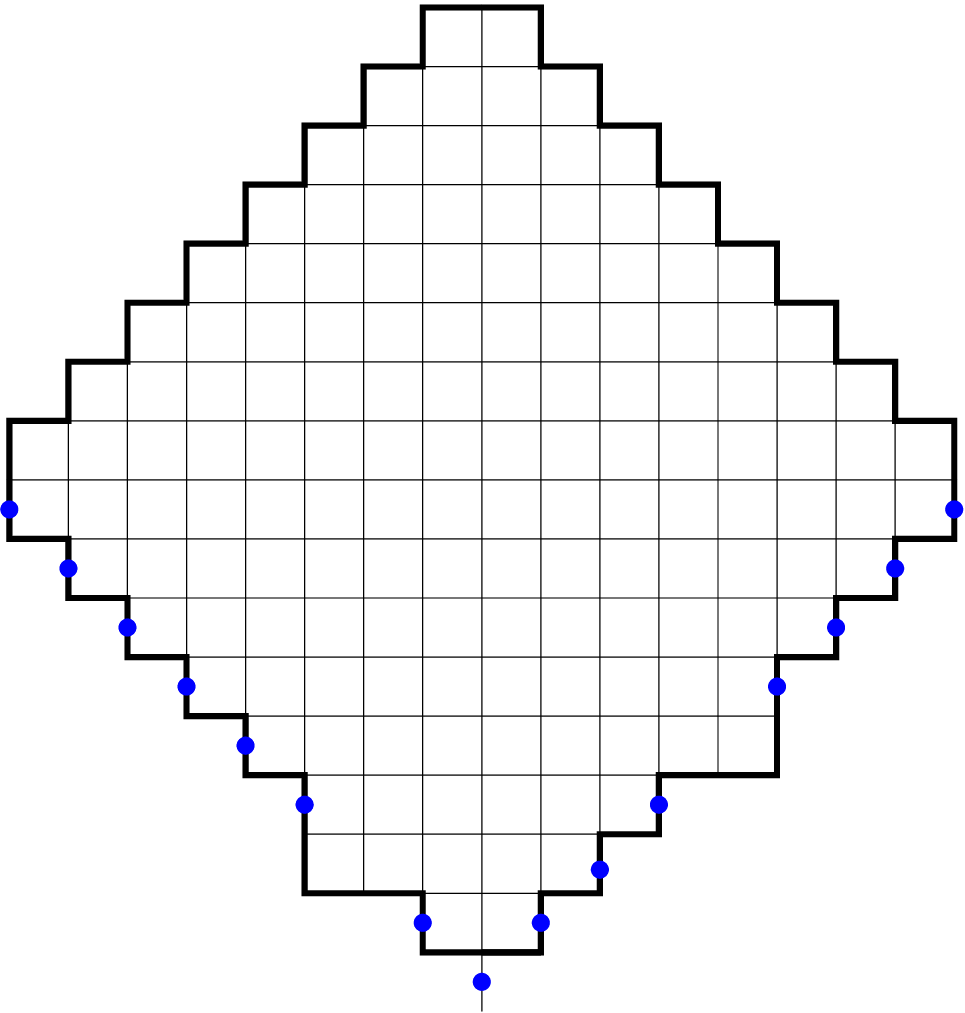}}
\hfill
}
\caption{\label{fcb} The region $\overline{AD}_{n}^{i,j}$ for $n=8$, $i=2$, $j=4$. The set of marked points on its boundary is the same as for $AD_n$ (see Figure~\ref{fca}), except the starting point labeled $i$ and the ending point labeled $j$ are omitted.}
\end{figure}

\begin{proof}[Proof of Theorem \ref{tba}]
Let $\overline{AD}_n^{\,i,j}$ be the region obtained from the Aztec diamond $AD_n$ by {\it adding} two unit squares: the one sharing an edge with the unit squares labeled $i$ and $i+1$ on the southwestern boundary, and the one sharing an edge with the unit squares labeled $j$ and $j+1$ on the southeastern boundary (see Figure~\ref{fcb}). We will first prove that
\begin{equation}
\frac{\M(\overline{AD}_n^{\,i,j})}{\M(AD_n)}
=
(-1)^{i+j}D_{n+1}^{-1}(i,j).
\label{ebo}
\end{equation}
Then we will deduce part (a) of Theorem~\ref{tba} from~\eqref{ebo} and Lemma~\ref{tcc}(a).

In order to prove~\eqref{ebo}, note that when one encodes the tilings of $\overline{AD}_n^{\,i,j}$ by families of non-intersecting Delannoy paths, the resulting Lindstr\"om--Gessel--Viennot matrix is simply\footnote{ Given a matrix $A$ with rows and columns labeled by $0,1\dotsc,n$, we denote by $M_{\{i_1,\dotsc,i_k\}}^{\{j_1,\dotsc,j_k\}}$ the submatrix obtained from $A$ by deleting rows $i_1,\dotsc,i_k$ and columns $j_1,\dotsc,j_k$.} $M_{\{i\}}^{\{j\}}$, where $M=D_{n+1}$ is the Lindstr\"om--Gessel--Viennot matrix obtained when encoding by Delannoy paths the tilings of the Aztec diamond $AD_n$. Indeed, this follows because the effect of adding the two unit squares to $AD_n$ to make the region $\overline{AD}_n^{\,i,j}$ is to remove the starting point labeled~$i$ and the ending point labeled $j$ from the boundary. By the Lindstr\"om--Gessel--Viennot theorem, this implies (using also that the Delannoy matrix is symmetric) that
\begin{equation}
\frac{\M(\overline{AD}_n^{\,i,j})}{\M(AD_n)}
=
\frac{\det (D_{n+1})_{\{i\}}^{\{j\}}}{\det D_{n+1}}
=
(-1)^{i+j}(D_{n+1})^{-1}(i,j),
\label{ebp}
\end{equation}
which proves \eqref{ebo}.

We will now show that 
\begin{equation}
\M(AD_{n+1}^{i+1,j+1})
=
2^n\M(\overline{AD}_{n}^{i,j}).
\label{ebq}
\end{equation}
Since $\M(AD_{n+1})=2^{n+1}\M(AD_n)$, this will imply by~\eqref{ebo} that
\begin{equation}
\frac{\M(AD_{n+1}^{i+1,j+1})}{\M(AD_{n+1})}
=
\frac12\frac{\M(\overline{AD}_{n}^{\,i,j})}{\M(AD_{n})}
=
\frac{(-1)^{i+j}}{2}D_{n+1}^{-1}(i,j).
\label{ebr}
\end{equation}
Note that part (a) of Theorem~\ref{tba} follows directly from~\eqref{ebr} and Lemma~\ref{tcc}(a), while part~(b) follows from~part (a) and
equation~\eqref{ebn}.

To prove \eqref{ebq}, recall that domino tilings of a region~$R$  on the square lattice can be identified with perfect matchings of the planar dual of~$R$. One readily verifies that, if we denote by~$G$ the planar dual graph of the region $AD_{n+1}^{i+1,j+1}$, then $G$ is a graph to which the complementation theorem of \cite[Theorem~2.1]{CT} can be applied, and the complement~$G'$ (in the sense of the quoted result) of $G$ is precisely the planar dual of $\overline{AD}_{n}^{i,j}$. It follows then by \cite[Theorem~2.1]{CT} that 
\begin{equation}
\M(AD_{n+1}^{i+1,j+1})
=
2^{n}\M(\overline{AD}_{n}^{i,j}),
\label{ebs}
\end{equation}
which proves~\eqref{ebq}. This completes the proof of Theorem~\ref{tba}. \end{proof}

More generally, for $1\leq i_1<\cdots<i_k\leq n$ and $1\leq j_1<\cdots<j_k\leq n$, define the dented Aztec diamond $AD_{n}^{i_1,\dotsc,i_k,j_1\dotsc,j_k}$ to be the region obtained from~$AD_n$ by removing the unit squares labeled $i_1,\dotsc,i_k$ on its southwestern boundary, and the ones labeled $j_1,\dotsc,j_k$ on its southeastern boundary. Then a determinant identity published by Jacobi\footnote{\label{foot:J}Jacobi's determinant identity (see Muir \cite[Eq.~(XX.4), p.~208]{Muir}) states that for any square matrix $M$ one has
\begin{equation*}
\det \left(\det M_{\{i_\mu\}}^{\{j^\nu\}}\right)_{1\leq \mu,\nu\leq k}
=
(\det M)^{k-1}\det M_{\{i_1,\dotsc,i_k\}}^{\{j_1,\dotsc,j_k\}}.
\end{equation*}
Provided $\det M\neq0$, this can be restated as
\begin{equation*}
\det \left(\frac{\det M_{\{i_\mu\}}^{\{j^\nu\}}}{\det M}\right)_{1\leq \mu,\nu\leq k}
=
\frac{\det M_{\{i_1,\dotsc,i_k\}}^{\{j_1,\dotsc,j_k\}}}{\det M},
\end{equation*}
which after transposing the matrix on the left-hand side gives
\begin{equation*}
\det \left((-1)^{i_\mu+j_\nu} M^{-1}(i_\mu,j_\nu)\right)_{1\leq \mu,\nu\leq k}
=
\frac{\det M_{\{i_1,\dotsc,i_k\}}^{\{j_1,\dotsc,j_k\}}}{\det M}.
\end{equation*}
} (see~\cite{Muir}) readily implies the following result. 

\begin{cor}
\label{tcd}
We have
\begin{equation}
\lim_{n\to\infty}\frac{\M(AD_{n}^{i_1,\dotsc,i_k,j_1\dotsc,j_k})}{\M(AD_n)}
=
\det \left(D(i_\mu-1,j_\nu-1)\right)_{1\leq \mu,\nu\leq k}.
\label{ebt}
\end{equation}
\end{cor}

\begin{proof} In analogy to $\overline{AD}_{n}^{\,i,j}$, define $\overline{AD}_{n}^{\,i_1,\dotsc,i_k;j_1\dotsc,j_k}$ to be the region obtained from $AD_n$ by adding to it the $k$ unit squares immediately to the left of the unit squares labeled $i_1,\dotsc,i_k$ on its southwestern boundary, and the $k$ unit squares immediately to the right of the unit squares labeled $j_1,\dotsc,j_k$ on its southeastern boundary. Encoding the tilings of $\overline{AD}_{n}^{\,i_1,\dotsc,i_k;j_1\dotsc,j_k}$ by families of non-intersecting Delannoy paths starting and ending at boundary points in the lower half of the region as indicated in Figure~\ref{fca}, one readily sees that the obtained Lindstr\"om--Gessel--Viennot matrix is precisely $B_{\{i_1,\dotsc,i_k\}}^{\{j_1,\dotsc,j_k\}}$, where $B=D_{n+1}$. Therefore we have
\begin{equation}
\M(\overline{AD}_{n}^{\,i_1,\dotsc,i_k;j_1\dotsc,j_k})
=
\det B_{\{i_1,\dotsc,i_k\}}^{\{j_1,\dotsc,j_k\}}.
\label{ebu}
\end{equation}
The argument that proved~\eqref{ebq} also implies
\begin{equation}
\M({AD}_{n+1}^{i_1+1,\dotsc,i_k+1;j_1+1\dotsc,j_k+1})
=
2^{n+1-k}
\M(\overline{AD}_{n}^{\,i_1,\dotsc,i_k;j_1\dotsc,j_k}),
\label{ebv}
\end{equation}
which in turn yields
\begin{equation}
\frac{\M(AD_{n+1}^{i_1+1,\dotsc,i_k+1;j_1+1\dotsc,j_k+1})}{\M(AD_{n+1})}
=
\frac{1}{2^k}\frac{\M(\overline{AD}_{n}^{\,i_1,\dotsc,i_k;j_1\dotsc,j_k})}{\M(AD_{n})}.
\label{ebw}
\end{equation}
However, since $B=D_{n+1}$ and $\M(AD_n)=\det B$, we obtain from~\eqref{ebu} and the restatement in footnote~\ref{foot:J} of Jacobi's determinant identity that
\begin{align}
\frac{\M(AD_{n+1}^{i_1+1,\dotsc,i_k+1;j_1+1\dotsc,j_k+1})}{\M(AD_{n+1})}
&=
\frac{1}{2^k}\frac{\det B_{\{i_1,\dotsc,i_k\}}^{\{j_1,\dotsc,j_k\}}}{\det B}
\nonumber
\\
&=
\frac{1}{2^k}
\det \left( (-1)^{i_\mu+j_\nu}B^{-1}(i_\mu\,j_\nu)\right)_{1\leq \mu,\nu\leq k}.
\label{ebx}
\end{align}
Since $B=D_{n+1}$, the statement of the corollary follows now from Lemma~\ref{tcc}(b). \end{proof}

\begin{rem} \label{rem:4}
The matrix on the right-hand side in~\eqref{ebt} is the Lindstr\"om--Gessel--Viennot matrix of a certain sub-region $R^{\,i_1,\dotsc,i_k;j_1\dotsc,j_k}$ of ${AD}_{n}^{\,i_1,\dotsc,i_k;j_1\dotsc,j_k}$ determined by the dents.
Therefore, the determinant on the right-hand side of~\eqref{ebt} is equal to $\M(R^{\,i_1,\dotsc,i_k;j_1\dotsc,j_k})$. In the special case when $\{i_1,\dotsc,i_k\}=\{1,\dotsc,k\}$, the region $R^{\,1,\dotsc,k;j_1\dotsc,j_k}$ is an Aztec rectangle with dents on the bottom, and by \cite[Eq.~(4.4)]{FT} its number of domino tilings is given by a simple product formula.
\end{rem}

\section{Proof of Theorem \ref{tbaa}}

The statement of Theorem~\ref{tbaa} follows directly from Theorem~\ref{tba}(a)
and part (a) of the following lemma.

\begin{lem}
\label{tda}  
$(${\rm a}$)$
As $n, x, y\to\infty$ so that $x/n \to a$ and $y/n \to b$, where $0 < a, b <1$ are fixed, we have
\begin{equation}
\lim_{n\to\infty}\frac{\displaystyle \sum_{k=0}^{n-1}{\binom k x}{\binom k y}\frac{1}{2^{k+1}}}{D(x,y)}
=
\begin{cases}
  1,  & \text{\rm if\ } (1-a)(1-b)>1/2, \\
  0,  & \text{\rm if\ } (1-a)(1-b)<1/2.
\end{cases}
\label{eda}  
\end{equation}  
$(${\rm b}$)$ As $n, x, y\to\infty$ so that $x/n \to a$ and $y/n \to b$, where $0 < a, b <1$ are fixed, the asymptotics of the numerator on the left-hand side above is given by
\begin{equation}
\sum_{k=0}^{n-1}{\binom k x}{\binom k y}\frac{1}{2^{k+1}}
\sim
\begin{cases}
  \frac
  {(a+b+\sqrt{a^2+b^2})\left(\frac{(a+\sqrt{a^2+b^2})^b(b+\sqrt{a^2+b^2})^a}{a^a b^b}\right)^n}
  {2\sqrt{2\pi n}\sqrt{ab\sqrt{a^2+b^2}}},  & \text{\rm if\ } (1-a)(1-b)>1/2, \\
  \\
  \frac{\sqrt{(1-a)(1-b)}\left(2a^ab^b(1-a)^{1-a}(1-b)^{1-b}\right)^{-n}}{4\pi n\sqrt{ab}\left(\frac12-(1-a)(1-b)\right)},  & \text{\rm if\ } (1-a)(1-b)<1/2.
\end{cases}
\label{edaa}  
\end{equation}

\end{lem}

Indeed, one can readily check that the line segment $[(a,0),(0,b)]$ is contained in the exterior of the circle inscribed in the unit square precisely when $(1-a)(1-b)>1/2$.

\begin{proof}[Proof of Lemma \ref{tda}]
Note that, by equation~\eqref{ebn}, the denominator on the left-hand side in~\eqref{tda} is $\sum_{k=0}^{\infty}{\binom k x}{\binom k y}\frac{1}{2^{k+1}}$.

We want to estimate
\begin{equation}
\sum_{k=0}^{n-1}\binom kx\binom ky2^{-k-1}
\end{equation}
as $n\to\infty$, where $x=an$ and $y=bn$. Let us write
\begin{align*}
F(a,b;n,k)&:=\binom k{an}\binom k{bn}2^{-k-1}\\
&=\frac {\Gamma^2(k+1)} {\Gamma(an+1)\,\Gamma(k-an+1)\,
\Gamma(bn+1)\,\Gamma(k-bn+1)}2^{-k-1}.
\end{align*}
In a first step, we want to determine the maximum of $F(a,b;n,nt)$
as a function in~$t$ (and with $a,b,n$ fixed). 
Obviously, we must compute the
derivative of $F(a,b;n,nt)$ with respect to~$t$ and
equate it to zero:
\begin{equation*}
\frac {\partial} {\partial t}\frac {\Gamma^2(nt+1)} {\Gamma(an+1)\,\Gamma(tn-an+1)\,
\Gamma(bn+1)\,\Gamma(tn-bn+1)}2^{-tn-1}=0.
\end{equation*}
Denoting the classical digamma function $\Gamma'(x)/\Gamma(x)$ by $\psi(x)$,
this leads to the equation
\begin{equation*}
2\psi(tn+1)-\psi(tn-an+1)-\psi(tn-bn+1)-\log 2=0.
\end{equation*}
Now, it is well known that $\psi(x)\sim\log x$ as $x\to\infty$. Hence, in a first approximation the above equation yields
\begin{equation*}
2\log(tn)-\log(tn-an)-\log(tn-bn)-\log2=0,
\end{equation*}
or, equivalently, 
\begin{equation*}
\frac {t^2} {(t-a)(t-b)}=2.
\end{equation*}
The solutions to this equation are $t=a+b\pm\sqrt{a^2+b^2}$, the maximum point
is $t=a+b+\sqrt{a^2+b^2}$. We have found that the maximum of $F(a,b;n,nt)$
occurs (roughly) at $t=a+b+\sqrt{a^2+b^2}$.

Now the question is whether this is smaller or larger than~1, meaning
whether the point of maximum is inside the summation range or not.
Hence, we must look at
\begin{equation*}
a+b+\sqrt{a^2+b^2}<1,
\end{equation*}
or, equivalently, at
\begin{equation*}
a^2+b^2<(1-a-b)^2,
\end{equation*}
or, again equivalently, at
\begin{equation}
\frac {1} {2}<(1-a)(1-b).
\label{tdone}
\end{equation}
Thus, if \eqref{tdone} is satisfied, then the point of maximum is in the interior of the summation, otherwise not.

\medskip
Let us first assume that~\eqref{tdone} holds, and write $t_0=a+b+\sqrt{a^2+b^2}$. Then, by Stirling's approximation in the form
\begin{align} 
\log\Ga(an+bl+c)&=\left(a n+b l+c-\frac {1} {2}\right)\left(\log\left(a+b \tfrac
{l\vphantom{1}} {n}\right)
+\log(n)+\log\left(1+\tfrac {c} {an+bl}\right)\right)
\nonumber
\\
&\kern1cm
        -(a n+b l+c)+
        \frac {1} {2}\log(2\pi)+O\left(\frac {1} {an+bl}\right)
\nonumber
\\
&=\left(a n+b l+c-\frac {1} {2}\right)(\log(a+b \tfrac ln)+\log(n))
\nonumber
\\
&\kern1cm
-(a n+b l)+
        \frac {1} {2}\log(2\pi)+O\left(\frac {1} {an+bl}\right),
\label{tdtwo}
\end{align}
we get
\begin{equation}
F(a,b;n,nt_0+l)=\frac {1} {2\pi\sqrt{
2ab}\,n}\exp\left(d(a,b)n
+e(a,b)\frac {l} {n}-f(a,b)\frac {l^2} {n}+O\left(\frac {l^3} {n^2}\right)\right),
\label{estar}
\end{equation}
where
\begin{align*}
d(a,b)&=-a \log a - b \log b + b \log\left(a + \sqrt{a^2 + b^2}\right) + 
 a \log\left(b + \sqrt{a^2 + b^2}\right),\\
e(a,b)&=-\frac{\sqrt{a^2+b^2}}{
   \left(a+b+\sqrt{a^2+b^2}\right)^2},\\
f(a,b)&=\frac{\sqrt{a^2+b^2}}
   {\left(\sqrt{a^2+b^2}+a+b\right)^2}
\end{align*}
In this computation, the (surprising) identity
\begin{equation*}
2\left(a+\sqrt{a^2+b^2}\right)\left(b+\sqrt{a^2+b^2}\right)
=\left(a+b+\sqrt{a^2+b^2}\right)^2
\end{equation*}
is used several times.

Now, in the sum
\begin{equation*}
\sum_lF(a,b;n,nt_0+l),
\end{equation*}
one restricts $l$ to $|l|<n^{3/5}$ (and such that $nt_0+l$ is an integer).
This has the
effect that this range captures the asymptotically relevant part of the
sum, while $\frac {l} {n}=O(n^{-2/5})$ and $\frac {l^3} {n^2}=O(n^{-1/5})$
tend to zero and are therefore asymptotically negligible, as is the
remaining sum (corresponding to the $l$'s outside this range;
this follows by the fact that, due to \eqref{estar}, at $l\sim\pm n^{3/5}$
we have $l^2/n\sim n^{1/5}$, implying that the corresponding summand is
exponentially small compared with the dominating terms, and, since the
summands with $|l|>n^{3/5}$ are even smaller, the sum over $|l|>n^{3/5}$
is negligible).
This leads to 
\begin{equation*}
\sum_{k=0}^{n-1}F(a,b;n,k)\sim 
\frac {e^{d(a,b)n}} {2\pi\sqrt{2ab}\,n}
\sum_{|l|<n^{3/5}}\exp\left(
-f(a,b)\frac {l^2} {n}\right).
\end{equation*}

Rewrite the sum on the right-hand side above as
\begin{equation*}
n^{1/2}
\sum_{|l|<n^{3/5}}n^{-1/2}\exp\left(
-f(a,b)\frac {l^2} {n}\right).
\end{equation*}
Now the sum (without the outer term $n^{1/2}$) is a Riemann sum
for the integral
\begin{equation*}
\int_{-n^{1/10}}^{n^{1/10}} \exp\left(-f(a,b)x^2\right) dx.
\end{equation*}

By replacing this integral by the integral from $-\infty$ to $\infty$, we make  an error that is exponentially small. Hence, as $n\to\infty$, the sum is asymptotically $n^{1/2}\int_{-\infty}^\infty \exp\left(-f(a,b)x^2\right)\,dx$. This yields
\begin{equation}
\sum_{k=0}^{n-1}F(a,b;n,k)\sim 
\frac {e^{d(a,b)n}} {2\sqrt{2\pi abf(a,b)n}},
\label{tdc}
\end{equation}
thus proving the asymptotics in the first branch of part (b).

Clearly, the above arguments also prove that the infinite sum $\sum_{k=0}^{\infty}F(a,b;n,k)$ has its $n\to\infty$ asymptotics given by the same expression on the right-hand side of~\eqref{tdc}. This,
together with~\eqref{ebn},
implies the statement in the first branch of part (a).

\medskip
If, on the other hand, we have $(1-a)(1-b)<\frac12$, then one
argues as follows. As we just said, the above arguments provide
the asymptotic approximation of
the complete sum $\sum_{k=0}^\infty F(a,b;n,k)$.
In particular, now the dominating part of the sum where $k$ is close
to $(a+b+\sqrt{a^2+b^2})n$ lies outside the range $0\le k\le n-1$.
This means at the same time that this range lies in the tail of
the complete sum that is negligible compared to the complete sum, thus proving the second branch in part (a). 

If we want to determine the exact asymptotics for the case where
$\frac {1} {2}>(1-a)(1-b)$, then we start by observing that the
summand in our sum is increasing in the whole range $0\le k\le n-1$.
It might  therefore be a good idea to reverse the order of summation,
and rewrite the sum as
\begin{align}
\sum_{k=0}^{n-1}\binom kx\binom ky2^{-k-1}
&=\sum_{k=0}^{n-1}F(a,b;n,k)
=\sum_{l=0}^{n-1}F(a,b;n,n-1-l)
\nonumber
\\
&=F(a,b;n,n-1)\sum_{l=0}^{n-1}\frac {F(a,b;n,n-1-l)} {F(a,b;n,n-1)}.
\label{tdthree}
\end{align}
Using again Stirling's formula in the
form~\eqref{tdtwo}, we get
\begin{align}
&
\frac {F(a,b;n,n-1-l)} {F(a,b,n,n-1)}=
\exp\left(l\cdot \log  \big(2(1-a)(1-b)\big)\right.
\nonumber
\\
&
\ \ \ \ \ \ \ \ \ \ \ \ \ \ \ \ \ \ \ \ \ \ \ \ \ \ \ \ \ \ \ \ \ 
\left.
+\left(\frac {l} {n}+\frac {l^2} {n}\right) \frac{2 a b-a-b}{2 (1-a) (1-b) }
   +O\left(\frac {l^3} {n^2}\right)\right)
\label{tdfour}
\end{align}
and
\begin{equation*}
F(a,b,n,n-1)=\frac {\sqrt{(1-a)(1-b)}} {2\pi\sqrt{ab}\, n}
\left(2a^ab^b(1-a)^{1-a}(1-b)^{1-b}\right)^{-n}
\left(1+O\left(\frac {1} {n}\right)\right).
\end{equation*}
Now it is easy to compute the asymptotics of the sum on the
right-hand side of~\eqref{tdthree}. Very roughly, it is
proportional to the geometric series
\begin{equation*}
\sum_{l=0}^{n-1}\big(2(1-a)(1-b)\big)^l,
\end{equation*}
which indeed converges to
\begin{equation*}
\sum_{l=0}^{\infty}\big(2(1-a)(1-b)\big)^l
=\frac {1} {1-2(1-a)(1-b)}
\end{equation*}
since we are in the case where $\frac {1} {2}> (1-a)(1-b)$.
If we do the fineprint, then the estimate~\eqref{tdfour} is not good enough for
the whole range of summation. Therefore, again, we first consider
a smaller range, here $0\le l<n^{2/5}$. In this range, we have
$\frac {l} {n}=O(n^{-3/5})$, $\frac {l^2} {n}=O(n^{-1/5})$,
and $O(\frac {l^3} {n^2})=O(n^{-4/5})$. Hence, these terms are
negligible for the summation in this range. Second, at $l=n^{2/5}$,
the summand is exponentially small compared to $F(a,b;n,n-1)$.
Since $F(a,b;n,n-l-1)$ is decreasing, the summands for $l>n^{2/5}$ are
even smaller. These are $O(n)$ summands, hence the contribution of the
range $n^{2/5}\le l\le n-1$ is exponentially smaller than the
contribution of the other summands. This yields
\begin{align*}
&
\sum_{k=0}^{n-1}\binom kx\binom ky2^{-k-1}
=\frac {\sqrt{(1-a)(1-b)}} {2\pi\sqrt{ab}\big(1-2(1-a)(1-b)\big) n}\\
&
\ \ \ \ \ \ \ \ \ \ \ \ \ \ \ \ \ \ \ \ \ \ \ \ \ \ \ \ \ \
\times
\left(2a^ab^b(1-a)^{1-a}(1-b)^{1-b}\right)^{-n}
\left(1+O\left(\frac {1} {n}\right)\right),
\end{align*}
which implies the asymptotic formula in the second branch in part (b). \end{proof}

The asymptotics of $\M(AD_n^{i,j})/\M(AD_n)$ follows immediately from Lemma~\ref{tda}. We obtain the following result, which gives a satisfying answer to the question that was the original motivation for this paper.

\begin{theo}
\label{tdb}
As $n,i,j\to\infty$ so that $i/n\to a$, $j/n\to b$, we have
\begin{equation}
\frac{\M(AD_n^{i,j})}{\M(AD_n)} 
\sim
\begin{cases}
  \frac
  {(a+b+\sqrt{a^2+b^2})\left(\frac{(a+\sqrt{a^2+b^2})^b(b+\sqrt{a^2+b^2})^a}{a^a b^b}\right)^n}
  {2\sqrt{2\pi n}\sqrt{ab\sqrt{a^2+b^2}}},  & \text{\rm if\ } (1-a)(1-b)>1/2, \\
  \\
  \frac{\sqrt{(1-a)(1-b)}\left(2a^ab^b(1-a)^{1-a}(1-b)^{1-b}\right)^{-n}}{4\pi n\sqrt{ab}\left(\frac12-(1-a)(1-b)\right)},  & \text{\rm if\ } (1-a)(1-b)<1/2.
\end{cases}
\label{edc}
\end{equation}

\end{theo}

The log-asymptotics of $\M(AD_n^{i,j})/\M(AD_n)$ follows immediately from Theorem~\ref{tdb}. We obtain the following result.

\begin{cor}
\label{cor:tdc}
As $n,i,j\to\infty$ so that $i/n\to a$, $j/n\to b$, we have
\begin{equation}
\frac{1}{n}\log\frac{\M(AD_n^{i,j})}{\M(AD_n)}\to f(a,b),
\label{edd}
\end{equation}
where the function $f(a,b)$ is defined by
\begin{equation}
f(a,b)=
\begin{cases}
  \log{\left(\dfrac{a}{b}+\sqrt{1+\dfrac{a^2}{b^2}}\right)^b\left(\dfrac{b}{a}+\sqrt{1+\dfrac{b^2}{a^2}}\right)^a},  & (1-a)(1-b)>\frac12, \\
\\
  -\log \,2a^ab^b(1-a)^{1-a}(1-b)^{1-b},  & (1-a)(1-b)<\frac12.
\end{cases}
\label{ede}
\end{equation}

\end{cor}

The graph of the function $f(a,b)$ is shown in Figure~\ref{fda}; it is reminiscent of a helmet.

\begin{figure}[t]
\vskip0.3in
\centerline{
\hfill
{\includegraphics[width=0.45\textwidth]{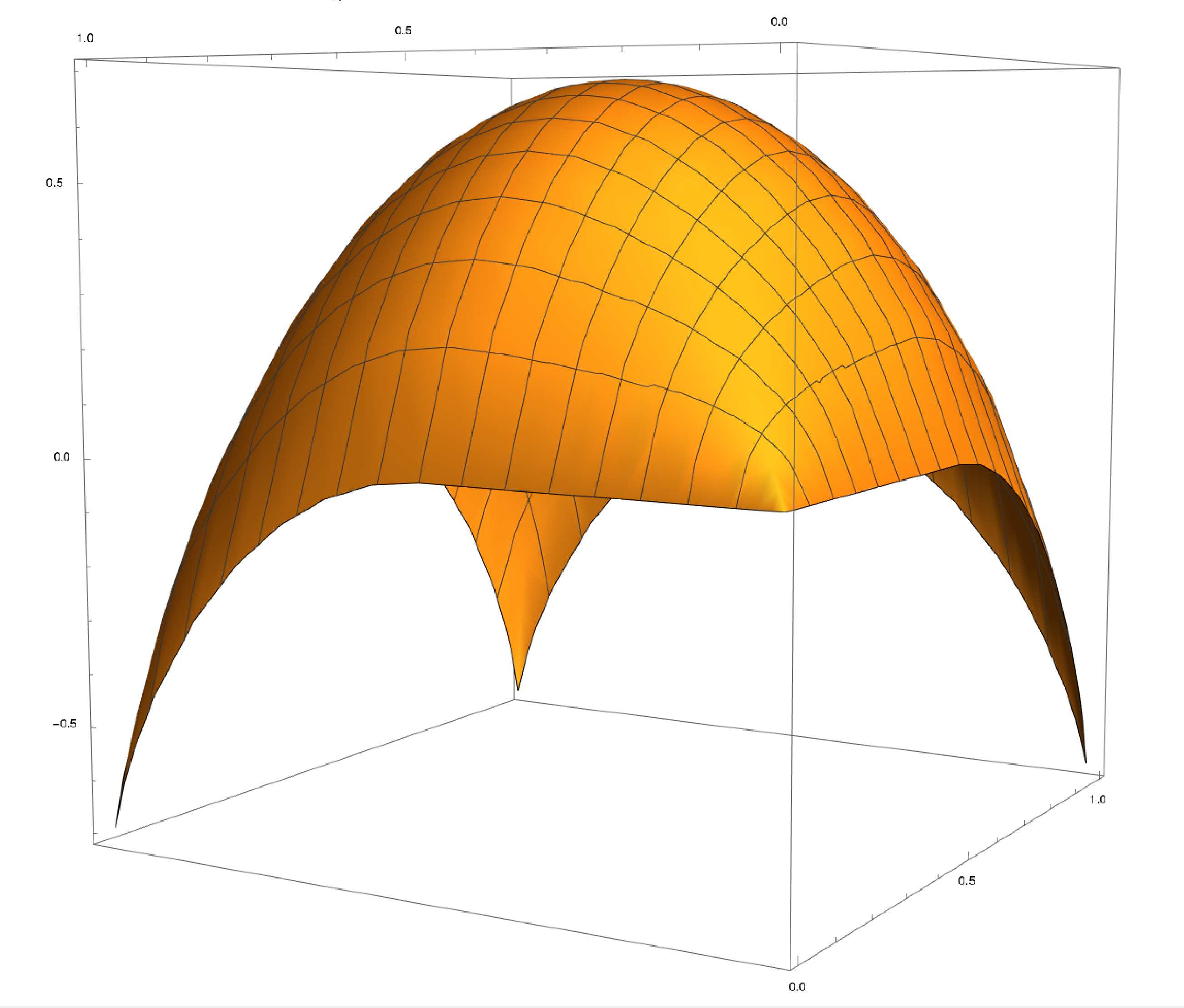}}
\hfill
{\includegraphics[width=0.37\textwidth]{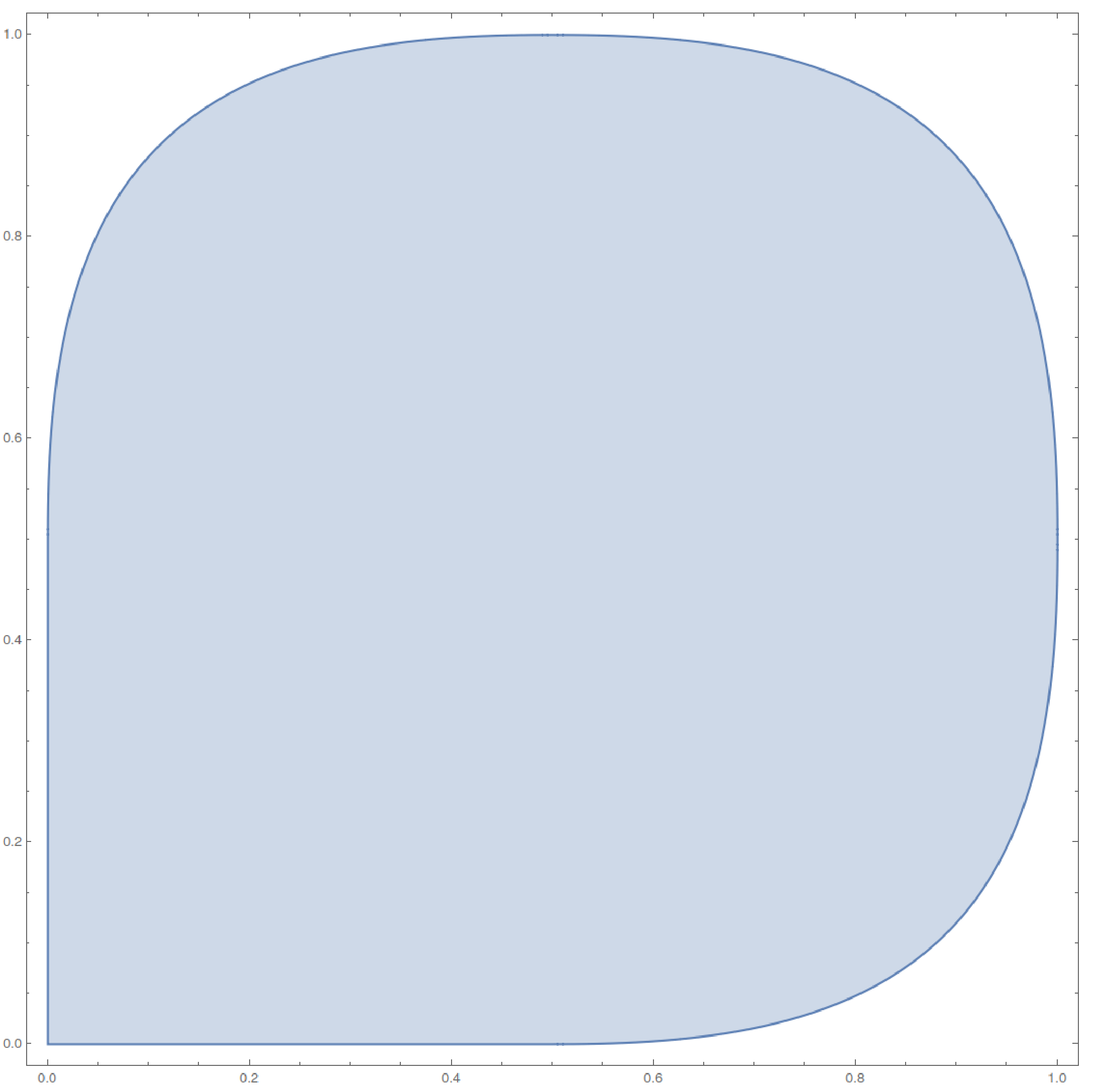}}
\hfill
}
\caption{\label{fda} {\it Left:} The limit shape of $\frac{1}{n}\log\frac{\M(AD_n^{i,j})}{\M(AD_n)}$ as $n,i,j\to\infty$ so that $i/n\to a$, $j/n\to b$. The near corner corresponds to $(a,b)=(0,0)$.  {\it Right:} The critical curve ${\mathcal C}$ along which $\M(AD_n^{i,j})$ and $\M(AD_n)$ have the same log-asymptotics (the boundary of the shaded region). Inside it the former has exponentially more tilings than the latter, and outside it exponentially fewer.}
\end{figure}

\begin{rem} \label{rem:5}
Note that at the three corners different from $(0,0)$ the ``helmet'' surface shown in Figure~\ref{fda} drops deeply below the value at $(0,0)$. This is not clear just from the electrostatic intuition --- it is an effect of the arctic circle phenomenon.
\end{rem}

\begin{rem} \label{rem:6}
The second branch in~\eqref{ede} is invariant under replacing $b$ by $1-b$. This implies that the cases $(a,b)$ and $(a,1-b)$, provided they are both in the second regime (i.e. $(1-a)(1-b)<1/2$, and $(1-a)(1-(1-b))=(1-a)b<1/2$), correspond to dented Aztec diamonds whose numbers of tilings have the same log-asymptotics, a non-trivial symmetry. In particular, the helmet surface in the neighborhood of $(a,b)=(1,1)$ is a mirror image of the helmet surface in the neighborhoods of $(a,b)=(1,0)$ and $(a,b)=(0,1)$ (the latter two are obviously mirror images of each other by the symmetry of the problem).
\end{rem}

\medskip
The answer to the question of how the numbers $\M(AD_n^{i,j})$ and $\M(AD_n)$ compare to each other is given in the following immediate consequence of Corollary~\ref{cor:tdc}.

\begin{cor}
\label{tdd}  
The domino tiling numbers $\M(AD_n^{i,j})$ and $\M(AD_n)$ have the same log-asympto\-tics as $n,i,j\to\infty$ so that $i/n\to a$, $j/n\to b$ if and only if $(a,b)$ is on the ``critical curve'' ${\mathcal C}$ obtained by taking the union of
$$
\{(a,b): 2a^ab^b(1-a)^{1-a}(1-b)^{1-b}=1,\ a>1/2\ \text{\rm or}\ b>1/2\}
$$
with the two line segments $[(0,0),(1/2,0)]$ and $[(0,0),(0,1/2)]$ $($${\mathcal C}$ is the boundary of the shaded region in the picture on the right in Figure~$\ref{fda}$$)$. For $(a,b)$ in the interior of this curve the dented Aztec diamond has exponentially more tilings than $AD_n$, and for $(a,b)$ outside ${\mathcal C}$ exponentially fewer.

\end{cor}

\begin{proof} The hyperbola $(1-a)(1-b)=1/2$ intersects the boundary of the unit square at the points $(1/2,0)$ and $(0,1/2)$. When the point $(a,b)$ is above it, we are in the regime $(1-a)(1-b)<1/2$. Setting the argument of the logarithm in the second branch in~\eqref{ede} equal to 1 we get the curvilinear portion of the curve ${\mathcal C}$. On the other hand, when the point $(a,b)$ is below the hyperbola, we are in the regime $(1-a)(1-b)>1/2$, and it is the first branch in~\eqref{ede} that gives the asymptotics.
The base in that exponential is greater than 1 for all $a,b>0$. If $a=0$ and $b<1/2$ (or $b=0$ and $a<1/2$), we are in the case $(1-a)(1-b)>1/2$ and it follows from \eqref{eba} and the arguments in the proof of Lemma \ref{tda} that $\M(AD_n^{i,j})$ and $\M(AD_n)$ have the same log-asympto\-tics.  
\end{proof}

Since the maximum of $f(a,b)$ occurs at $(a,b)=(1/2,1/2)$, it follows that, at least asymptotically, the maximum value of the number of domino tilings of the dented Aztec diamonds $AD_n^{i,j}$ occurs when $a=b=1/2$. This also agrees with the electrostatic intuition --- each dent balances itself at equal distances from the two sides for which the space outside them acts as a huge charge of the same sign as the dent.

\begin{theo}
\label{tde}
Let $0<a_1<\cdots<a_k<1$ and  $0<b_1<\cdots<b_k<1$ be fixed, and assume that the segment $[(a_i,0),(0,b_j)]$ crosses the circle inscribed in the unit square, for all $0\leq i,j\leq k$.
Then as $n,i_1,\dotsc,i_k,j_1,\dotsc,j_k\to\infty$ so that $i_1/n\to a_1,\dotsc,i_k/n\to a_k,j_1/n\to b_1,\dotsc,j_k/n\to b_k$, the asymptotics of the ratio between the number of domino tilings of the dented and undented Aztec diamonds is given by

\begin{align}
&
  \frac{\M(AD_n^{i_1,\dotsc,i_k;j_1,\dotsc,j_k})}{\M(AD_n)}
  \sim
  \nonumber
  \\
&\ \ \ \ \ \ \ \ \ \ 
  \frac{2^{-k(k-1)/2}}{4\pi n}
  \prod_{i=1}^k\frac{\sqrt{\frac{1}{a_i}-1}\sqrt{\frac{1}{b_i}-1}}
       {\left(2a_i^{a_i}(1-a_i)^{1-a_i}b_i^{b_i}(1-b_i)^{1-b_i}\right)^n}
       \frac{\prod_{1\leq i<j\leq k}(a_j-a_i)(b_j-b_i)}
            {\prod_{i=1}^k\prod_{j=1}^k\frac12-(1-a_i)(1-b_j)}.
\label{edf}
\end{align}  

\end{theo}  

\begin{proof} By the proof of Corollary~\ref{tcd} we have (replacing $n$ by $n-1$ and subtracting 1 from the index of each dent)
\begin{align}
\frac{\M(AD_n^{i_1,\dotsc,i_k;j_1,\dotsc,j_k})}{\M(AD_n)}
&=
\frac{1}{2^k}\det \left((-1)^{i_s+j_t}D_n^{-1}(i_s-1,j_t-1)\right)_{1\leq s,t\leq k}
\nonumber
\\
&=\det \left(\sum_{l=0}^{n-1}{\binom l {i_s-1}}{\binom l {j_t-1}}\frac{1}{2^{l+1}}\right)_{1\leq s,t\leq k}
\label{edg}
\end{align}
where at the second equality we used Lemma~\ref{tcc}(a). By the assumptions in the statement of the theorem, the asymptotics of each entry in the matrix on the right-hand side of~\eqref{edg} is given by the second branch in Lemma~\ref{tda}(b). The determinant obtained by replacing each entry by this asymptotics is easily seen to evaluate to the expression on the right-hand side of~\eqref{edf} --- and since this is clearly non-zero, it gives the asymptotics of the determinant on the right-hand side of~\eqref{edg}; this proves~\eqref{edf}. Indeed, to obtain the evaluation of the determinant, note that the expression
\begin{equation*}
\prod_{s=1}^k\frac{\sqrt{1-a_s}}{4\pi n\sqrt{a_s}}\left(2a_s^{a_s}(1-a_s)^{1-a_s}\right)^{-n}
\end{equation*}
can be pulled out as a common factor from the entries of row $s$, for $s=1,\dotsc,k$. Similarly, the expression
\begin{equation*}
\prod_{t=1}^k\frac{\sqrt{1-b_s}}{\sqrt{b_t}}\left(b_t^{b_t}(1-b_t)^{1-b_t}\right)^{-n}
\end{equation*}
can be pulled out as a common factor from the entries of column $s$, for $s=1,\dotsc,k$. The leftover determinant is 
\begin{equation*}
\det\left(\frac{1}{\frac12-(1-a_s)(1-b_t)}\right)_{1\leq s,t\leq k}.
\end{equation*}
However, a formula for this readily follows from the
evaluation of the Cauchy determinant,
\begin{equation}
\det\left(\frac{1}{1-u_i v_j}\right)_{1\leq i,j\leq k}=
\frac{\prod_{1\leq i<j\leq k}(u_j-u_i)(v_j-v_i)}{\prod_{i=1}^k\prod_{j=1}^k 1-u_iv_j}.  
\label{edh}
\end{equation}
Combining these leads to the expression on the right-hand side of~\eqref{edg}. \end{proof}

\begin{cor}
\label{tdf}
Under the same assumptions as in Theorem~{\em\ref{tde}}, the limit of the difference in entropy per site\footnote{ The entropy of a dimer system on the graph $G$ is $\log \M(G)$. If $G$ has $n$ vertices, the entropy per site is~$(1/n)\log \M(G)$.} between the dented and undented Aztec diamonds is given by

\begin{align}
  \lim_{n\to\infty}\frac{1}{n}\log\frac{\M(AD_n^{i_1,\dotsc,i_k;j_1,\dotsc,j_k})}{\M(AD_n)}
  =-\sum_{i=1}^k\log 2a_i^{a_i}(1-a_i)^{1-a_i}b_i^{b_i}(1-b_i)^{1-b_i}.
\label{edff}
\end{align}  

\end{cor}  

\begin{rem} \label{rem:7}
The limiting entropy difference in Corollary~\ref{tdf} can be thought of as the scaling limit of a correlation of the dents. Then Corollary~\ref{tdf} indicates that, when all segments connecting the dents cross the inscribed circle, the scaling limit of this correlation of the dents turns out to be determined by the individual interactions of each dent with the two corners bounding the side on which the dent is --- the quantity $d^{-d}$ is contributed by each such dent-corner pair (where $d$ is the distance from the dent to the corner), and these are simply multiplied together to get the limiting value of the correlation (up to the additive constant $-k\log2$, which could be regarded as the ``background'' entropy per site of the system). 

It would be interesting to work out the correlation of the dents in the general case, when the segments connecting the dents need not cross the inscribed circle.
\end{rem}

\section{Lozenge tilings}

There is a natural analog of the question that motivated the above results for lozenge tilings of hexagons. Let $H_{a,b,c}$ be the hexagonal region on the triangular lattice whose sides have lengths $a$, $b$, $c$, $a$, $b$, $c$ (clockwise from top). A lozenge is the union of two unit triangles that share an edge. A lozenge tiling of a region~$R$ is a covering of $R$ by lozenges, with no gaps or overlaps. Given a hexagonal region $H_{a,b,c}$, which unit up-pointing triangle and which unit down-pointing triangle should one remove from it so that the remaining region has maximum number of lozenge tilings?

Based on the electrostatic intuition described in the introduction, one would expect this maximum to be achieved when each removed unit triangle is at the middle of an appropriate edge --- southern, northwestern or northeastern for the former, and one of the remaining ones for the latter. 

Recall that the Pochhammer symbol $(a)_k$ is defined for $k\geq0$ by
$(a)_k=a(a+1)\cdots(a+k-1)$,
and the ${}_rF_s$ hypergeometric series of parameters $a_1,\ldots,a_r$ and $b_1,\ldots,b_s$ by
$$
\pFq{r}{s}{a_1,\ldots,a_r}{b_1,\ldots,b_s}{z} =
\sum_{k=0}^{\infty} 
\frac{(a_1)_k \cdots (a_r)_k}{(b_1)_k \cdots (b_s)_k} \frac{z^k}{k!}.
$$

\begin{figure}[t]
\vskip0.15in
\centerline{
\hfill
{\includegraphics[width=0.35\textwidth]{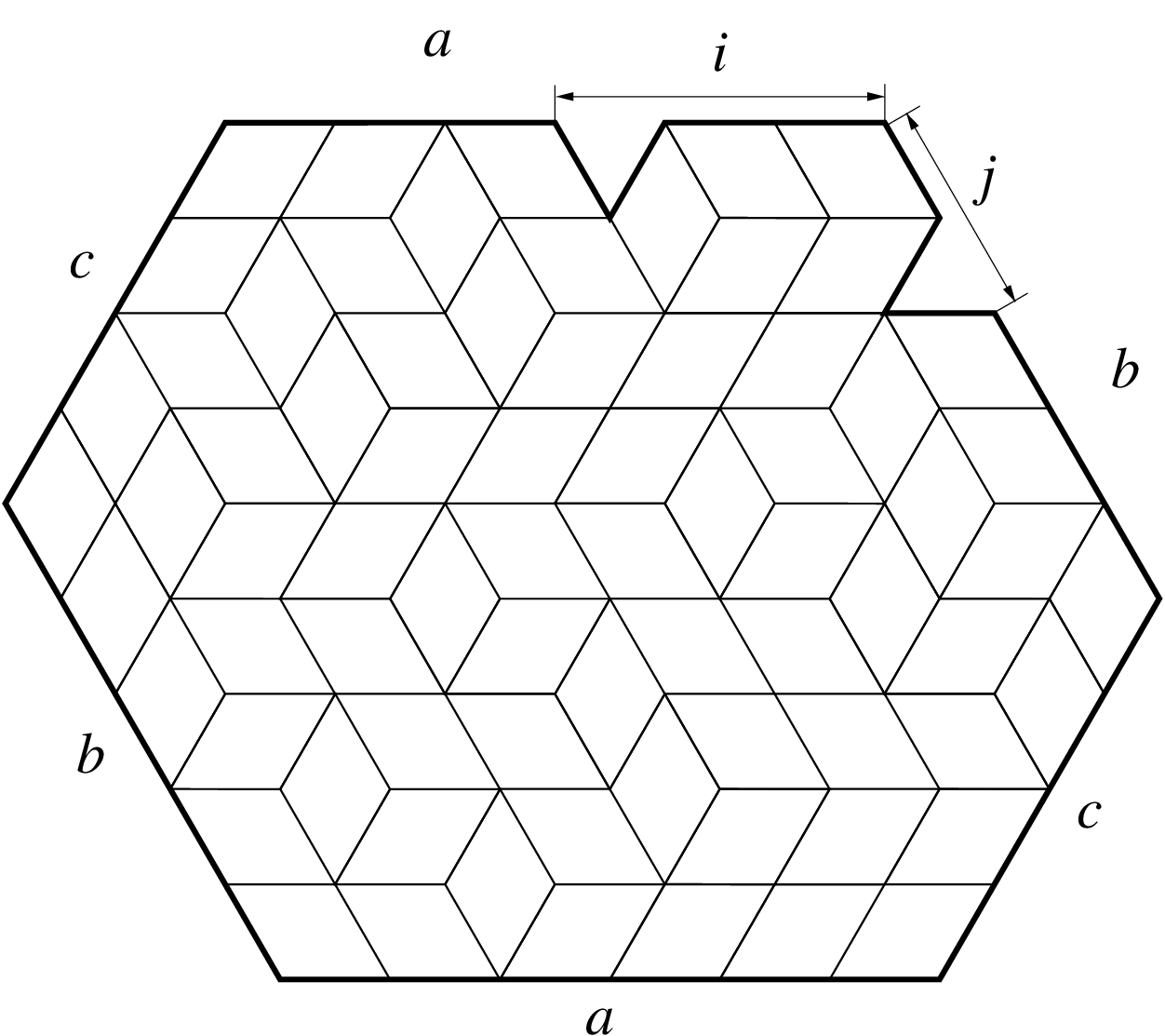}}
\hfill
{\includegraphics[width=0.35\textwidth]{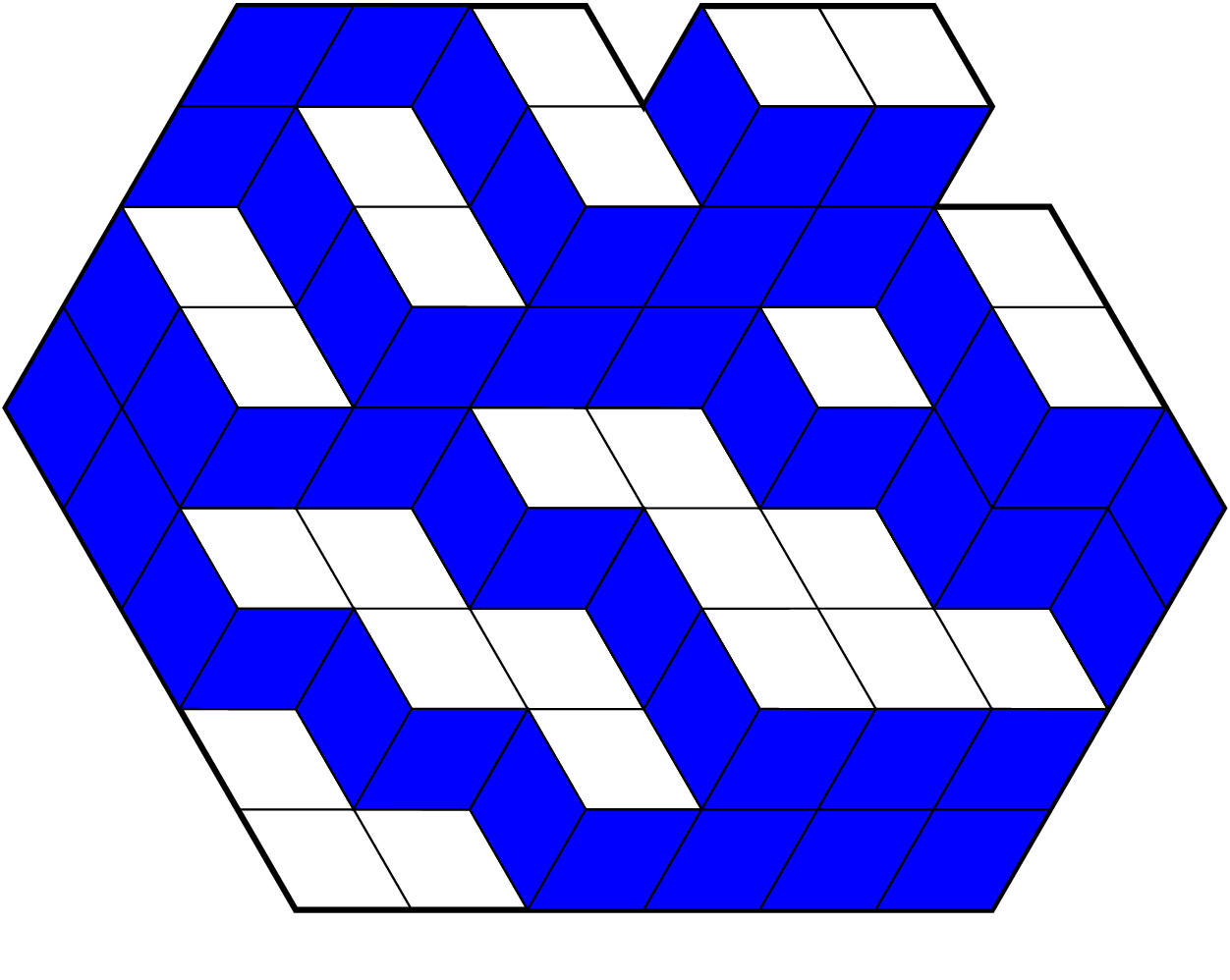}}
\hfill
}
\caption{\label{fea} {\it Left:} A tiling of the dented hexagon $H_{a,b,c}^{i,j}$ for $a=6$, $b=5$, $c=4$, $i=3$, $j=2$. {\it Right:} The family of non-intersecting paths of lozenges encoding the tiling.}
\end{figure}

Let $H_{a,b,c}^{i,j}$ be the dented hexagon obtained from $H_{a,b,c}$ by removing the $i$th unit triangle from along its top side (counted from right to left) and the $j$th unit triangle from along its northeastern side (counted from top to bottom); see Figure~\ref{fea} for an example. The following result is an immediate consequence of \cite[Proposition~3]{boundarydents}.

\begin{theo}
\label{tea}
For any $1\leq i\leq a$, $1\leq j\leq b$, we have:
\begin{equation}
\frac{\M(H_{a,b,c}^{i,j})}{\M(H_{a,b,c})}=
\frac{(b)_a}{(b+c)_a}
\frac{(1+c)_{a-i} (i)_{j-1} (1+b-j)_{j-1}}{(1)_{a-i} (1)_{j-1} (1+b+c-j)_{j-1}}
\pFq{3}{2}{-a+i,-b+j,c}{1-a-b, 1+c}{1}. 
\label{eea}
\end{equation}

\end{theo}

\begin{proof} Use the formula in \cite[Proposition~3]{boundarydents} and the fact that, by MacMahon's boxed plane partition formula~\cite{MacM}, $\M(H_{a,b,c})=\prod_{i=0}^{a-1}\frac{(b+i+1)_c}{(1+i)_c}$. \end{proof}

We are interested in the asymptotic behavior of $\M(H_{a,b,c}^{i,j})/\M(H_{a,b,c})$ as $a,b,c,i,j\to\infty$ so that $a\sim An$, $b\sim Bn$, $c\sim Cn$ for some fixed $A,B,C>0$, and $i\sim \al An$, $j\sim \be Bn$ for some fixed $0<\alpha,\beta<1$.
Draw the dented hexagons $H_{a,b,c}^{i,j}$ so that they are centered at the origin, and as the parameters tend to infinity as specified rescale them by a homothety of factor $1/n$ through the origin; then the rescaled dented hexagons approach a hexagon $H$ of side-lengths $A$, $B$, $C$, $A$, $B$, $C$ (clockwise from top) centered at the origin, with two marked points $P_0$ and $Q_0$ on the northern and northeastern sides corresponding to the location of the dents.
The following analog of Theorem~\ref{tbaa} holds.

\begin{theo}
\label{teb}

Let $E$ be the ellipse inscribed in $H$, and $P_0$ and $Q_0$ the scaling limits of the positions of the dents on the northern and northeastern sides of $H$, respectively.
Then as $a,b,c,i,j\to\infty$ as described in the previous paragraph, we have 
\begin{equation}
\lim_{n\to\infty}\frac{\M(H_{a,b,c}^{i,j})}{{\binom {i+j-2} {j-1}}\M(H_{a,b,c})}=
\begin{cases}
  1,  & \text{\rm if the segment $[P_0,Q_0]$ is outside $E$}, \\
  0,  & \text{\rm if the segment $[P_0,Q_0]$ crosses $E$}.
\end{cases}
\label{eeb}  
\end{equation}

\end{theo}

\begin{rem} \label{rem:8}
The arguments presented in Remark~\ref{rem:3} in Section~2 apply equally well for deducing that the arctic curve for lozenge tilings of the hexagons $H_{a,b,c}$ is the ellipse $E$ inscribed in $H$.

As in Remark~\ref{rem:3}, in order to make our argument, we need to assume that an arctic curve for lozenge tilings of $H_{a,b,c}$ exists. More precisely, assume that there is a convex arc ${\mathscr E}$ tangent to the northern and northeastern sides of the hexagon $H$ so that if $\widehat{\mathscr E}$ is the closed region below ${\mathscr E}$ enclosed between ${\mathscr E}$ and the boundary of $H$ we have:

\medskip
$(i')$ for any open set $U$ containing $\widehat{\mathscr E}$, the family of $c$ non-intersecting paths of lozenges corresponding to a lozenge tiling\footnote{\label{foot:L} It is well-known that lozenge tilings of any simply connected region on the triangular lattice are in one-to-one correspondence with families of non-intersecting paths of lozenges. For instance, the family of non-intersecting paths of lozenges corresponding to the lozenge tiling on the left in Figure~\ref{fea} is shown on the right in the same figure.} of $H_{a,b,c}$ connecting the sides of length $c$ is contained in $U$ with probability approaching 1, as $n\to\infty$; on the other hand, for any open set $V\subsetneq \widehat{\mathscr E}$, with probability approaching 1 this family of $n$ non-intersecting lozenge paths is {\it not} contained in $V$.

\medskip
Then, using exactly the same reasoning as in Remark~\ref{rem:3} (and replacing fact $(ii)$ by its analog for paths with steps $(1,0)$ and $(0,1)$, which follows as a limiting case of \cite[Theorem~1]{BBE},
and also from the more general result in the Appendix of the present article) we obtain that, under the assumption $(i')$, Theorem~\ref{teb} implies that the arc ${\mathscr E}$ must coincide with the shorter of the two arcs of the ellipse $E$ inscribed in $H$ bounded by the points where it touches the northern and northeastern sides of $H$. By symmetry, this in turn implies that the arctic curve for lozenge tilings of $H_{a,b,c}$ is the ellipse $E$ inscribed in $H$.
\end{rem}    

\begin{rem} \label{rem:9}
In the special case when $i$ and $j$ are fixed (so $\alpha=\beta=0$), it follows as a limiting case of Theorem~\ref{teb} that $\M(H_{a,b,c}^{i,j})/\M(H_{a,b,c})\to{\binom {i+j-2} {j-1}}$. This provides an example when by removing some portion of a region the number of tilings goes up $k$ times (in the limit), for any desired positive integer $k$ (simply take $i=k$, $j=2$).
\end{rem}

In the proof of Theorem~\ref{teb} we will employ the following analog of Lemma~\ref{tda}(b). Its proof is presented after Remark~\ref{rem:10} (which follows Corollary~\ref{tee}).

\begin{lem}
\label{tec}
Let $a,b,c,i,j\to\infty$ so that
$a\sim An$, $b\sim Bn$, $c\sim Cn$ for some fixed $A,B,C>0$, and $i\sim \al An$, $j\sim \be Bn$ for some fixed $0<\alpha,\beta<1$.

$(${\rm a}$)$. Then, provided
\begin{equation}
(1-\al)(1-\be)AB-(\al A+\be B)C>0,
\label{eec}
\end{equation}
we have
\begin{align}
&
\pFq{3}{2}{ c,-b+j, -a+i}{ c+1,1-a-b}{1}
\sim
 \sqrt{
\frac {2\pi n(1 -  \al)A(1 - \be)B(A+B)C }
{\big((1 -  \al)A + C\big)\big((1-  \be)B + C\big)(A+B+C)} }
 \nonumber
 \\[5pt]
 &\ \ \ \ \ \ \ \ \ \ \ \ \ \ 
 \times
 \left(\frac {\big((1-  \al)A\big)^{(1-  \al)A}\big((1- \be)B\big)^{(1- \be)B}
(A+B+C)^{A+B+C}C^C}
{(A+B)^{A+B}\big((1 -  \al)A + C\big)^{(1 -  \al)A + C}
\big((1 - \be) B +  C\big)^{(1 - \be) B +  C} }\right)^n.
\label{eed}
\end{align}

$(${\rm b}$)$. If on the other hand $(1-\al)(1-\be)AB-(\al A+\be B)C<0$, then we have

\begin{align}
&
\pFq{3}{2}{ c,-b+j, -a+i}{ c+1,1-a-b}{1}
\sim
 \sqrt{
\frac {(A+B)(\al A+\be B)^3}
{(\al A+B)(A+\be B)} }
\nonumber
 \\[5pt]
 &\ \ \ \ \ \ \ \ \ \ \ \ \ \ 
 \times
 \frac{C}{(\al A+\be B)C-(1-\al)(1-\be)AB}
 \left(\frac{(\al A+B)^{\al A+B}(A+\be B)^{A+\be B}}{(A+B)^{A+B}(\al A+\be B)^{\al A+\be B}}\right)^n.
\label{eee}
\end{align}

 \end{lem}

\begin{theo}
\label{ted}
Let $E$ be the ellipse inscribed in $H$, and $P_0$ and $Q_0$ the scaling limits of the
positions of the dents on the northern and northeastern sides of H, respectively.
Then as $a,b,c,i,j\to\infty$ so that $a\sim An$, $b\sim Bn$, $c\sim Cn$
for some $A,B,C>0$, and $i/a\to\al$, $j/b\to\be$ for some $0<\al,\be<1$, we have
\begin{align}
&
\frac{\M(H_{a,b,c}^{i,j})}{\M(H_{a,b,c})}\sim
\nonumber
\\
&\ \ \ \ \ \ \ \ 
\begin{cases}
  \frac{1}{\sqrt{2\pi n}}\sqrt{\frac{\al A\,\be B}{(\al A+\be B)^3}}
  \left(\frac{(\al A+\be B)^{\al A+\be B}}{(\al A)^{\al A}(\be B)^{\be B}}\right)^n,  & \text{\rm if $[P_0,Q_0]$ is outside $E$}, \\
  \\
  \frac{1}{2\pi n}\frac{1}{(\al A+\be B)C-(A-\al A)(B-\be B)}\\[10pt]
  \times\sqrt{\frac{\al A\,\be B(A-\al A+C)(B-\be B+C)(A+B+C)C}{(A-\al A)(B-\be B)(\al A+B)(A+\be B)}}\\[10pt]
  \times \left(\frac{(A-\al A+C)^{A-\al A+C}(B-\be B+C)^{B-\be B+C}(\al A+B)^{\al A+B}(A+\be B)^{A+\be B}}
  {(\al A)^{\al A}(A-\al A)^{A-\al A}(\be B)^{\be B}(B-\be B)^{B-\be B}(A+B+C)^{A+B+C}C^C}\right)^n,  & \text{\rm if $[P_0,Q_0]$ crosses $E$}.
\end{cases}
\label{eef}  
\end{align}

\end{theo}

\begin{proof}
It readily follows from Stirling's formula~\eqref{tdtwo} that for $a\sim An$, $b\sim Bn$, $c\sim Cn$, $i\sim\al A n$ and $j\sim \be Bn$, the quantity by which the $_3 F_2$-series in equation~\eqref{eea} is multiplied has asymptotics
\begin{align}
&
\frac{1}{2\pi n}\sqrt{\frac{\al A \,\be B (A-\al A+C)(B-\be B+C)(A+B+C)}
{(A-\al A)(B-\be B)(\al A+\be B)^3(A+B)C}}
\nonumber
\\[10pt]
&\ \ \ \ \ \ \ \ 
\times
\left(\frac{(\al A+\be B)^{\al A+\be B}(A-\al A+C)^{A-\al A+C}(B-\be B+C)^{B-\be B+C}(A+B)^{A+B}}
 {(\al A)^{\al A}(A-\al A)^{A-\al A}(\be B)^{\be B}(B-\be B)^{B-\be B}(A+B+C)^{(A+B+C)}C^C}\right)^n.
\label{eeg}
\end{align}
Note also that --- as it is not hard to verify --- the segment $[P_0,Q_0]$ is contained in the exterior of the ellipse $E$ inscribed in the scaling limit hexagon $H$ if and only if inequality~\eqref{eec} holds. Therefore, Theorem~\ref{tea} combined with Lemma~\ref{tec} and equation~\eqref{eeg} yields formulas~\eqref{eef}.
\end{proof}

\begin{cor}
\label{tee}
As $a,b,c,i,j\to\infty$ so that $a\sim An$, $b\sim Bn$, $c\sim Cn$
for some $A,B,C>0$, and $i/a\to\al$, $j/b\to\be$ for some $0<\al,\be<1$, we have

%
\begin{align}
&
\log\frac{\M(H_{a,b,c}^{i,j})}{\M(H_{a,b,c})}=
\nonumber
\\[10pt]
&\ \ 
\begin{cases}
  n\log \dfrac{(i+j)^{i+j}}{i^ij^j}+O(\log n),  & (a-i)(b-j)\gtrsim(i+j)c, \\
\\
  n\log \dfrac{(a-i+c)^{a-i+c}(b-j+c)^{b-j+c}(i+b)^{i+b}(a+j)^{a+j}}
       {i^i(a-i)^{a-i}j^{j}(b-j)^{b-j}(a+b+c)^{a+b+c}c^c}\\[10pt]
       \ \ \ \ \ \ \ \ \ \ \ \ \ \ \ \ \ \ \ \ \ \ \ \ \ \ \ \ \ \ \ \ \ \ \ \ \ \ \ \ \ \ \ \ \ \ \ \ \ \ \ \ \
       +O(\log n),  & (a-i)(b-j)\lesssim(i+j)c,
\end{cases}
\label{eei}
\end{align}
where the asymptotic inequalities $(a-i)(b-j)\gtrsim(i+j)c$ and $(a-i)(b-j)\lesssim(i+j)c$ stand for $(A-\al A)(B-\be B)>(\al A+\be B)C$ and $(A-\al A)(B-\be B)<(\al A+\be B)C$, respectively.

\end{cor}

\begin{proof} This follows directly from Theorem~\ref{ted}, using that in the asymptotics under consideration we have $(\al A+\be B)n\sim i+j$, etc. \end{proof} 

\begin{rem} \label{rem:10}
The formulas in Corollary~\ref{tee} can be interpreted as follows.
The quantities $\log \M(H_{a,b,c}^{i,j})$ and $\log \M(H_{a,b,c})$ are the entropies of the dented and undented hexagons, respectively. Thus,  Corollary~\ref{tee} implies that, when $(A-\al A)(B-\be B)-(\al A+\be B)C>0$, we have
\begin{align}
\lim_{n\to\infty}\frac{1}{n}\log\frac{\M(H_{a,b,c}^{i,j})}{\M(H_{a,b,c})}
&=
(\al A+\be B)\log (\al A+\be B) -\al A \log \al A -\be B\log\be B
\nonumber
\\
&=F(\al A)+F(\be B)-F(\al A+\be B),
\label{eex2}
\end{align}
where $F(x)=-x\log x$. On the other hand,  if  $(A-\al A)(B-\be B)-(\al A+\be B)C<0$, we obtain from Corollary~\ref{tee} that
\begin{align}
\lim_{n\to\infty}\frac{1}{n}\log\frac{\M(H_{a,b,c}^{i,j})}{\M(H_{a,b,c})}
=
&F(A+B+C)+F(C)
\nonumber
\\
&+\sum_{v\, {\rm dent},\, C\, {\rm neigbor\, or\, next\, neighbor\, corner}}(-1)^{\#\, \text{\rm corners\, between\,} v\, {\rm and\,} C}F(\de(v,C)),
\label{eey}
\end{align}
where $v$ ranges over the set $\{P_0,Q_0\}$ consisting of the scaling limits of the positions of the two dents on the northern and northeastern sides of the scaling hexagon $H$, $C$ ranges over the four corners of $H$ that are either incident to the side $e$ containing $v$, or are incident to a side next to $e$, and $\de$ is the distance calculated around the perimeter of the hexagon.


It is quite remarkable that in both cases the limiting entropy difference can be expressed so simply in terms of the function $F(x)=-x\log x$ --- the same function that is used for defining the classical statistical mechanics entropy of a system (in that context, the entropy is defined, up to a multiplicative constant, to be $-\sum_ip_i\log p_i$, where $p_i$ is the probability that the system is in the $i$th state). For dented Aztec diamonds, this only holds when the segment joining the dents crosses the inscribed circle (see formulas~\eqref{ede}).
\end{rem}


\begin{proof}[Proof of Lemma \ref{tec}]
We want to estimate
\begin{equation}
\pFq{3}{2}{ c,-b+j, -a+i}{c+1,1-a-b}{1}
\label{eeone}
\end{equation}
as $n\to\infty$, where $a\sim An$, $b\sim Bn$, $c\sim Cn$, $i\sim \al A n$, and $j\sim \be Bn$. To start with, we apply the transformation formula
(see \cite[Ex.~7, p.~98, terminating form]{BailAA})
\begin{equation*}
\pFq{3}{2}{ a, b, -N}{d, e}{1}
  =
\frac{(-a - b + d + e)_N}{(e)_N}
\pFq{3}{2}{-N, -a + d, -b + d}{d, -a - b + d + e}{1},
\end{equation*}
where $N$ is a non-negative integer.
Thus, the hypergeometric series in~\eqref{eeone} becomes
\begin{align}
&
\frac {(2 - a - j)_{a-i}} {(1 - a - b)_{a-i}}
\pFq{3}{2}{ -a + i, 1, 1 + b + c - j}{1 + c, 2 - a - j}{1}
\nonumber
\\
&
=
\frac {\,\Ga(c+1) \,\Ga(a - i+1) \,\Ga(b + i) }
{\Ga(a + b) \,\Ga(b + c - j+1) \,\Ga(i + j-1) }
\sum_{k=0}^{a-i}
\frac {\Ga(a + j - k-1) \,\Ga(b + c - j + k+1)}
{\Ga(a - i - k+1) \,\Ga(c + k+1)}.
\label{eeoneandahalf}
\end{align}

Hence, our task is to estimate the asymptotics of the sum over~$k$,
under the above limit scheme.
Let
\begin{align*}
F(A,B,C,\al,\be;n,k):=&\frac {\,\Ga(Cn+1) \,\Ga(An - \al An+1) \,\Ga(Bn + \al An) }
{\Ga(An + Bn) \,\Ga(Bn + Cn - \be Bn+1) \,\Ga(\al An + \be Bn-1) }
\nonumber
\\
&\times
\frac {\Ga(An + \be Bn - k-1) \,\Ga(Bn + Cn - \be Bn + k+1)}
{\Ga(An - \al An - k+1) \,\Ga(Cn + k+1)}.
\end{align*}

In a first step, we want to determine the maximum of $F(A,B,C,\al,\be;n,nt)$
as a function in~$t$ (and with $A,B,C,\al,\be,n$ fixed). 
Obviously, we must compute the
derivative of $F(A,B,C,\al,\be;n,nt)$ with respect to~$t$ and
equate it to zero:
\begin{equation*}
\frac {\partial} {\partial t}
\frac {\Ga(An + \be Bn - tn-1) \,\Ga(Bn + Cn - \be Bn + tn+1)}
{\Ga(An - \al An - tn+1) \,\Ga(Cn + tn+1)}=0.
\end{equation*}
This leads to the equation
\begin{align*}
  -\ps(An + \be Bn - tn-1) &+\ps(Bn + Cn - \be Bn + tn+1)
\nonumber
  \\
&+\ps(An - \al An - tn+1) -\ps(Cn + tn+1)=0,
\end{align*}
where, as in the proof of Lemma~\ref{tda}, $\ps(x)$ is the classical digamma function. 

Hence, using again that $\ps(x)\sim\log x$ as $x\to\infty$, in a first
approximation the above equation yields
\begin{align*}
  -\log(An + \be Bn - tn-1) &+\log(Bn + Cn - \be Bn + tn+1)
\nonumber
  \\
&+\log(An - \al An - tn+1) -\log(Cn + tn+1)=0,
\end{align*}
which is equivalent to 
\begin{equation*}
\frac {(Bn + Cn - \be Bn + tn+1)(An - \al An - tn+1)}
{(An + \be Bn - tn-1)(Cn + tn+1)}=1.
\end{equation*}
Since we are interested in the large~$n$ case, asymptotically this leads
to the equation
\begin{equation*}
\frac {(B + C - \be B + t)(A - \al A - t)}
{(A + \be B - t)(C + t)}=1,
\end{equation*}
whose solution is
%
\begin{equation*}
t=\frac {(1-\al)(1-\be)A B  - (\al A  + \be B) C} {\al A + B}.
\end{equation*}
This should be in the range $0<t<A-\al A$. Thus, the lower bound
inequality, $0<t$, says that
\begin{equation}
(1-\al)(1-\be)A B  - (\al A  + \be B) C>0,
\label{eetwo}
\end{equation}
while the upper bound inequality,
$t<A-\al A$, is equivalent to
\begin{equation*}
\frac {(\al A + \be B) (A - \al A + C)} {\al A+ B}
 >0,
\end{equation*}
which is always true since $A,B,C,\al,\be>0$ and $\al<1$.

\medskip
We now concentrate on the case where~\eqref{eetwo} holds, that is, where the maximum
of the summand occurs inside the summation range.
We write
\begin{equation*}
t_0=\frac {(1-\al)(1-\be)A B  - (\al A  + \be B) C} {\al A + B}.
\end{equation*}
Then, by Stirling's approximation in the form~\eqref{tdtwo}
we get
\begin{align*}
&
F(A,B,C,\al,\be;n,nt_0+l)=
\frac {\sqrt{(1 -  \al)A(A+B)(\al A + B)^3C}}
{\big((1 -  \al)A + C\big)(A+B+C)\sqrt{(\al A + \be B)\big((1-  \be)B + C\big)}}
\\
&
\times
\big(d(A,B,C,\al,\be)\big)^n
\exp\left(
e(A,B,C,\al,\be)\frac {l} {n}
-f(A,B,C,\al,\be)\frac {l^2} {n}+O\left(\frac {l^3} {n^2}\right)\right),
\end{align*}
where
\begin{align*}
d(A,B,C,\al,\be)&=
\frac {\big((1-  \al)A\big)^{(1-  \al)A}\big((1- \be)B\big)^{(1- \be)B}
(A+B+C)^{A+B+C}C^C}
{(A+B)^{A+B}\big((1 -  \al)A + C\big)^{(1 -  \al)A + C}
  \big((1 - \be) B +  C\big)^{(1 - \be) B +  C} },
\nonumber
\\
e(A,B,C,\al,\be)&=
\frac {\begin{matrix}
(A \al + B) (- \al^2A^2 + 4 A B - 4  \al A B + B^2 
\kern3.5cm\\\kern3cm
 - 4  \be A B + 2  \al \be A B - 2  \be B^2 + 4 B C - 4 \be B C)\end{matrix}}
      {2 (1 - \be) B(\al A + \be B) ((1 -  \al)A + C) (A + B + C) },
\nonumber
\\
f(A,B,C,\al,\be)&=
\frac {(\al A + B)^3 }
{2 (1 - \be)B (\al A + \be B) ((1 -  \al)A + C) (A + B + C) }.
\end{align*}

Now, in the sum
\begin{equation*}
\sum_lF(A,B,C,\al,\be;n,nt_0+l),
\end{equation*}
one restricts $l$ to $|l|<n^{3/5}$ (and such that $nt_0+l$ is an integer).
As
in the proof of Lemma~\ref{tda},
this has the
effect that this range captures the asymptotically relevant part of the
sum, while $\frac {l} {n}=O(n^{-2/5})$ and $\frac {l^3} {n^2}=O(n^{-1/5})$
tend to zero and are therefore asymptotically negligible, as is the
remaining sum (corresponding to the $l$'s outside this range).
This leads to 
\begin{align*}
\sum_{k=0}^{a-i}F(A,B,C,\al,\be;n,k)&\sim 
\frac {\sqrt{(1 -  \al)A(A+B)(\al A + B)^3C}}
{\big((1 -  \al)A + C\big)(A+B+C)\sqrt{(\al A + \be B)\big((1-  \be)B + C\big)}}
\nonumber
\\
&\times \big(d(A,B,C,\al,\be)\big)^n
\sum_{|l|<n^{3/5}}\exp\left(
-f(A,B,C,\al,\be)\frac {l^2} {n}\right).
\end{align*}
The rest of the argument is the same as in the proof of Lemma~\ref{tda}.
One can approximate the sum by the integral
\begin{equation*}
\int_{-\infty}^\infty\exp\left(
-f(A,B,C,\al,\be)\frac {l^2} {n}\right)\,dl,
\end{equation*}
and one gets
%
\begin{align}
\sum_{k=0}^{a-i}F(A,B,C,\al,\be;n,k)&\sim 
\frac {\sqrt{(1 -  \al)A(A+B)(A \al + B)^3C}}
{\big((1 -  \al)A + C\big)(A+B+C)\sqrt{(\al A + \be B)\big((1-  \be)B + C\big)}}
\nonumber
\\
&\times \frac {\sqrt{\pi n}} {\sqrt{f(A,B,C,\al,\be)}}
\big(d(A,B,C,\al,\be)\big)^n
\nonumber
\\
&\!\!\!\!\!\!\!\!\!\!\!\!\!\!\!\!\!\!\!\!\!\!\!\!\!\!\!\!\!\!\!\!
\sim \sqrt{
\frac {2\pi n(1 -  \al)A(1 - \be)B(A+B)C }
{\big((1 -  \al)A + C\big)\big((1-  \be)B + C\big)(A+B+C)} }
\big(d(A,B,C,\al,\be)\big)^n.
\label{eethreeandahalf}
\end{align}
This proves part (a) of Lemma~\ref{tec}.

Now let the inequality~\eqref{eetwo} hold the opposite way, that is,
\begin{equation*}
(1-\al)(1-\be)A B  - (\al A  + \be B) C<0.
\end{equation*}
Then the maximum of $F(A,B,C,\al,\be;n,k)$ occurs for $k<0$, and therefore
the summand $F(A,B,C,\al,\be;n,k)$ is decreasing for $0\le k\le a-i$.
Using again Stirling's formula in the form~\eqref{tdtwo}, we get
\begin{align}
&
F(A,B,C,\al,\be;n,l)=\sqrt{\frac {(A + B)(\al A + \be B)^3}
{(\al A + B)(A + \be B)^3}}
\left(\frac {(\al A+B)^{\al A+B}(A+\be B)^{A+\be B}}
{(A+B)^{A+B}(\al A+\be B)^{\al A+\be B}}
\right)^{n}
\nonumber
\\
&\ \ \ \ \ \ \ \ 
\times
\left(
\frac {(1-\al)A((1-\be)B+C)} {(A+\be B)C}\right)^l
\exp\left(O\left(\frac {l} {n}\right)+O\left(\frac {l^2} {n^2}\right)
+O\left(\frac {l^3} {n^2}\right)\right).
\label{eefour}
\end{align}
As in the earlier considerations,
in the sum, we would have to restrict $l$ to $0\le l<n^{2/5}$.
The terms for $l>n^{2/5}$ would be negligible. On the other hand, 
the sum of the terms~\eqref{eefour} over all $l\ge0$ is a geometric series which yields
\begin{align*}
\pFq{3}{2}{c,-b+j, -a+i}{ c+1,1-a-b}{1}
&\sim \sqrt{\frac {(A + B)(\al A + \be B)^3}
{(\al A + B)(A + \be B)}}
\left(\frac {(\al A+B)^{\al A+B}(A+\be B)^{A+\be B}}
{(A+B)^{A+B}(\al A+\be B)^{\al A+\be B}}
\right)^{n}
\nonumber
\\
&\times
\frac {C} {(\al A  +\be  B) C-(1-\al)(1-\be)A B },
\quad \text{as $n\to\infty$.}
\end{align*}
This completes the proof of Lemma~\ref{tec}.
\end{proof}

\begin{proof}[Proof of Theorem \ref{teb}]
Stirling's formula~\eqref{tdtwo} readily implies that if $i\sim \al An$ and  $j\sim \be Bn$,
\begin{equation}
{\binom {i+j-2} {i-1}}\sim\frac{1}{\sqrt{2\pi n}}\sqrt{\frac{\al A\,\be B}{(\al A+\be B)^3}}
\left(\frac{(\al A+\be B)^{\al A+\be B}}{(\al A)^{\al A}(\be B)^{\be B}}\right)^n.
\label{eefive}
\end{equation} 
%
oThe first branch of~\eqref{eeb} follows directly by combining~\eqref{eefive} with the first branch of~\eqref{eef}.

Suppose now that $(1-\al)(1-\be)A B  - (\al A  + \be B) C<0$. The statement will follow provided we show that the base of the exponential in the second branch of~\eqref{eef} is strictly smaller than the base of the exponential in~\eqref{eefive}. This amounts to proving that
%
\begin{equation}
  \frac{[(1-\al)A+C]^{(1-\al)A+C}[(1-\be)B+C]^{(1-\be)B+C}(\al A+B)^{\al A+B}(A+\be B)^{A+\be B}}
       {[(1-\al)A]^{(1-\al)A}[(1-\be)B]^{(1-\be)B}(A+B+C)^{A+B+C}C^C(\al A+\be B)^{\al A+\be B}}<1.
\label{eesix}
\end{equation} 
Denote by $U$ and $V$ the bases of the exponentials in Lemma \ref{tec}:
\begin{align*}
U&=\frac{[(1-\al)A]^{(1-\al)A}[(1-\be)B]^{(1-\be)B}(A+B+C)^{A+B+C}C^C}
       {(A+B)^{A+B}[(1-\al)A+C]^{(1-\al)A+C}[(1-\be)B+C]^{(1-\be)B+C}},
\\[10pt]
V&=\frac{[(\al A+B)^{\al A+B}(A+\be B)^{A+\be B}}
       {(A+B)^{A+B}(\al A+\be B)^{\al A+\be B}}.
\end{align*} 
Then the left hand side of \eqref{eesix} is just $V/U$, and proving \eqref{eesix} amounts to proving that $V<U$.

To see this, note that the arguments in the proof of Lemma \ref{tec} also prove that, when replacing the upper limit of the sum on the left hand side of \eqref{eethreeandahalf} by infinity, the resulting infinite sum has asymptotics given by the right hand side of \eqref{eethreeandahalf} --- and that this holds irrespective of the sign of $(1-\al)(1-\be)AB-(\al A+\be B)C$.

Suppose this sign is negative. Then the base of the exponential in the asymptotics of the infinite sum is still $U$. On the other hand, as we saw in the proof of Lemma \ref{tec}, in this case the sum restricted to the range $0\leq k\leq a-i$ lies in the tail of the complete sum, so it is negligible compared to the latter. Since the base of the exponential in the asymptotics of the former is~$V$, this shows $V<U$, and the proof is complete. \end{proof}

\medskip
As mentioned before (see footnote~\ref{foot:L}), lozenge tilings of regions on the triangular lattice are encoded in a one-to-one fashion by families of non-intersecting paths of lozenges, which in turn can be identified with lattice paths on $\Z^2$ for which the allowed steps are $(1,0)$ and $(0,1)$. Encoding lozenge tilings by such families of non-intersecting paths, using the Lindstr\"om--Gessel--Viennot theorem and applying Jacobi's determinant identity as in Corollary~\ref{tcd} we obtain the following result.

\begin{cor}
\label{tef}
Let $H_{a,b,c}^{i_1,\dotsc,i_k,j_1\dotsc,j_k}$ be the dented hexagon obtained from $H_{a,b,c}$ by removing the down-pointing unit triangles at fixed distances $i_1,\dotsc,i_k$ from the top right corner from along its northern side, and the up-pointing unit triangles at fixed distances $j_1,\dotsc,j_k$ from the top right corner from along its northeastern side. Then we have
\begin{equation}
\lim_{n\to\infty}\frac{\M(H_{a,b,c}^{i_1,\dotsc,i_k,j_1\dotsc,j_k})}{M(H_{a,b,c})}
=
\det \left({\binom {i_\mu+j_\nu-2} {j_\nu-1}}\right)_{1\leq \mu,\nu\leq k}.
\label{eel2}
\end{equation}
\end{cor}  

It turns out that Theorem~\ref{tde} has an analog for lozenge tilings of hexagons. More precisely, when the positions of the dents are such that each segment joining dents on different sides crosses the inscribed ellipse, the joint correlation of $k$ dents on the northern side and $k$ on the northeastern side has asymptotics given by a simple product.

\begin{theo}
\label{teg}
Let $A,B,C>0$, $0<\al_1<\cdots<\al_k<1$ and  $0<\be_1<\cdots<\be_k<1$ be fixed. Assume that $(1-\al_s)(1-\be_t)AB-(\al_s A+\be_t B)C<0$ for all $0\leq s,t\leq k$ $($i.e.\ the segment joining the scaling limit of the position of each $\alpha$-dent with the scaling limit of the position of each $\be$-dent crosses the ellipse inscribed in the hexagon $H$$)$.
Then as $a,b,c,i_1,\dotsc,i_k,j_1,\dotsc,j_k\to\infty$ so that $a\sim An$, $b\sim Bn$, $c\sim Cn$ for some fixed $A,B,C>0$ and $i_s\sim \al_s An$, $j_t\sim \be_t Bn$, the asymptotics of the ratio between the numbers of domino tilings of the dented and undented hexagons is given by
\begin{align}
&
  \frac{\M(H_{a,b,c}^{i_1,\dotsc,i_k;j_1,\dotsc,j_k})}{\M(H_{a,b,c})}
  \sim
  \left(\frac{1}{2\pi n}\right)^{k}((A+B+C)C)^{nk}\left(\frac{1}{(A+B+C)^{A+B+C}C^C}\right)^{nk/2}
  \nonumber
  \\[10pt]
&\ \ \ \ \ \ \ \ \ \ 
\times  \prod_{s=1}^k\sqrt{\frac{\al_s A(A-\al_s A+C)}{(A-\al_s A)(\al_s A+B)}}
  \left(\frac{(A-\al_s A+C)^{A-\al_s A+C}(\al_s A+B)^{\al_s A+B}}{(\al_s A)^{\al_s A}(A-\al_s A)^{A-\al_s A}}\right)^n
  \nonumber
  \\[10pt]
&\ \ \ \ \ \ \ \ \ \ 
\times \prod_{t=1}^k\sqrt{ \frac{\be_t B(B-\be_t B+C)}{(B-\be_t B)(A+\be_t B)}}
  \left(\frac{(B-\be_t B+C)^{B-\be_t B+C}(A+\be_t B)^{A+\be_t B}}{(\be_t B)^{\be_t B}(B-\be_t B)^{B-\be_t B}}\right)^n
  \nonumber
  \\[10pt]
&\ \ \ \ \ \ \ \ \ \ 
\times \left(\frac{1}{-AB}\right)^k\left(\frac{C(A+B+C)}{AB}\right)^{\binom k 2} 
       \frac{\prod_{1\leq s<t\leq k}(\al_t-\al_s)(\be_t-\be_s)}
            {\prod_{i=1}^k\prod_{j=1}^k 1-(1+C/B)\al_s-(1+C/A)\be_t+\al_s\be_t}.
\label{eek}
\end{align}  
\end{theo}

\begin{proof} By the determinant identity in
footnote~\ref{foot:J}
and the Lindstr\"om--Gessel--Viennot theorem we have
\begin{equation}
\frac{\M(H_{a,b,c}^{i_1,\dotsc,i_k;j_1,\dotsc,j_k})}{\M(H_{a,b,c})}
=
\det \left(\frac{\M(H_{a,b,c}^{i_s,j_t})}{\M(H_{a,b,c})}\right)_{1\leq s,t\leq k}.
\label{eel}
\end{equation}
Replace each entry in the matrix on the right-hand side above by its asymptotics, given by the second branch in Theorem~\ref{ted}. By pulling out common factors along the rows and columns of the resulting matrix --- yielding the factors in the first three lines on the right-hand side of equation~\eqref{eek} --- evaluating the determinant on the right-hand side of~\eqref{eel} amounts to determining
%
\begin{equation}
\det \left(\frac{1}{(\al_s A+\be_t B)C-(1-\al_s)(1-\be_t)AB}\right)_{1\leq s,t\leq k}.
\label{eem}
\end{equation}
However, this follows from the Cauchy matrix-like identity
\begin{equation}
\det \left(\frac{1}{1+x a_i+y b_j +z a_i b_j}\right)_{1\leq i,j\leq k}=(xy-z)^{\binom k 2}\frac{\prod_{1\leq i<j\leq k}(a_j-a_i)(b_j-b_i)}
{\prod_{i=1}^k\prod_{j=1}^k 1+xa_i+yb_j+za_ib_j}.
\label{een}
\end{equation}
Indeed, the latter identity is implied by the evaluation of the Cauchy
determinant in~\eqref{edh} by the choice of $u_i=(xy-z)^{-1/2}(y+z a_i)$
and $v_i=(xy-z)^{-1/2}(x+z b_i)$, $i=1,2,\dots,k$, and subsequent simplification.
This identity implies, as one can readily see, that the determinant in~\eqref{eem} is equal to the expression on the fourth line on the right-hand side of~\eqref{eek}\footnote{ Note that $1-(1+C/B)\al_s-(1+C/A)\be_t+\al_s\be_t$ equals $1/AB$ times $(1-\al_s)(1-\be_t)AB-(\al_s A+\be_t B)C$, and is therefore negative by assumption; these negative signs combine with the sign of $(-1/AB)^k$ to yield an overall positive sign.}. 
\end{proof}

\begin{rem} \label{rem:11}
Why is the asymptotics of the joint correlation of dents given by a simple product in the segment-crossing-ellipse case, but not in general? This seems to be a consequence of the arctic ellipse phenomenon. It would be interesting to understand more clearly the reason for this.
\end{rem}

\begin{cor}
\label{teh}
Under the same assumptions as in Theorem~\ref{teg}, the limit of the difference in entropy per site between the dented and undented hexagons is given by

%
\begin{align}
&
  \lim_{n\to\infty}\frac{1}{n}\log\frac{\M(H_{a,b,c}^{i_1,\dotsc,i_k;j_1,\dotsc,j_k})}{\M(H_{a,b,c})}
\nonumber
\\[10pt]
&\ \ \ \ \ \ 
=\sum_{s=1}^k\log \dfrac{(A-\al_s A+C)^{A-\al_s A+C}(B-\be_s B+C)^{B-\be_s B+C}(\al_s A+B)^{\al_s A+B}(A+\be_s)^{A+\be_s}}
       {{(\al_s A)}^{\al_s A}(A-\al_s A)^{A-\al_s A}{(\be_s B)}^{\be_s B}(B-\be_s B)^{B-\be_s B}(A+B+C)^{A+B+C}C^C}.
\label{eeo}
\end{align}  

\end{cor}

\begin{rem} \label{rem:12}
In view of the question in Remark~\ref{rem:11}, Corollary~\ref{teh} indicates that, as a result of the distortion of dimer statistics caused by the arctic ellipse phenomenon, the asymptotics of the joint correlation of the dents is determined by summing up all the interactions of the individual dents with the corners of the hexagon (see Remark~\ref{rem:10}) --- which, in this segment-crossing-ellipse case, completely swamp any pairwise interaction between the dents.

Note also that \eqref{eeo} implies that formula~\eqref{eey}, with $v$ running over the set of $2k$ dents and $F(A+B+C)+F(C)$ replaced by $kF(A+B+C)+kF(C)$, holds also in this case. This allows an elegant reformulation of~\eqref{eeo} as a ``superposition principle:'' The joint correlation of the dents at positions $\{i_1,\dotsc,i_k,j_1,\dotsc,j_k\}$ is the sum of the correlations of the $k$ pairs of dents at positions $\{i_s,j_s\}$, for $s=1,\dotsc,k$.
\end{rem}

Turning back to the question of the position of the two removed unit triangles from $H_{a,b,c}$ that make the number of tilings of the leftover region maximum, provided one of $A,B,C$ is strictly larger than the other two, the electrostatic intuition suggests that the pair of parallel sides corresponding to $\max(A,B,C)$ will have the dominant effect, and end up ``attracting'' the two removed unit triangles to their midpoints. Numerical data also overwhelmingly support this.

An interesting special case arises when $A=B=C$. Then the above intuition singles out two non-equivalent positions: dented hexagons with the dents at the middle of two adjacent sides, and dented hexagons with the dents at the middle of two opposite sides. It turns out that there is a strikingly simple relationship between them.

\begin{theo}
\label{tem}
Let $\widehat{H}_{n,n,n}^{\lfloor n/2\rfloor,\lfloor n/2\rfloor}$ be the region obtained from the hexagon $H_{n,n,n}$ by making unit dents in the middle of two opposite sides. Then we have
\begin{equation}
  \lim_{n\to\infty}\frac{\M(H_{n,n,n}^{\lfloor n/2\rfloor,\lfloor n/2\rfloor})}
      {\M(\widehat{H}_{n,n,n}^{\lfloor n/2\rfloor,\lfloor n/2\rfloor})}=2.
\label{eex}  
\end{equation}

\end{theo}

Note that by contrast, if rather than being dents on the boundary of the regular hexagon these two pairs of removed unit triangles would be in the bulk,
the ratio of the two pair correlations would be $\sqrt{2}$ (cf.\ \cite[Eq.~(6)]{ov}, as in a regular hexagon the distance between the midpoints of opposite sides is precisely 2 times the distance between midpoints of neighboring sides, and the charges of the removed triangles are $1$ and $-1$).

\begin{proof}[Proof of Theorem \ref{tem}]
According to \cite[Proposition~4]{boundarydents}, we have                      
\begin{align*}
\frac {\M(\hat H_{a,b,c}^{i,j})} {\M(H_{a,b,c})}
=\frac {(a)_b\,(c)_{j-1}\,(1+b-j)_{i-1}\,(2+a-i-j)_{i+j-2}}
  {(a+c)_b\,(i-1)!\,(j-1)!\,(1+a+c-i)_{i-1}\,(1+a+b-j)_{j-1}}
\nonumber
\\[10pt]
\times
  {} _{4} F _{3} \!\left [ \begin{matrix} { 1 - j, 1 - c - j, 1 + a + b - j, 1 - i}
\nonumber
\\
{ 1 + b - j, 2 + a - i - j, 2 - c - j }\end{matrix} ; {\displaystyle 1}\right ].
\end{align*}
To this $_4F_3$-series, we apply one of Sears' transformations for balanced
$_4F_3$-series, namely (cf.\ \cite[Eq.~(4.3.5.1)]{Slater})
\begin{align*}
&
  {} _{4} F _{3} \!\left [ \begin{matrix} { a, b, c, -n}
\nonumber
  \\
  { e, f, 1 + a + b + c - e -
  f - n}\end{matrix} ; {\displaystyle 1}\right ]
\nonumber
\\[10pt]
&\ \ \ \ \ \ \ \ 
  =
{\frac {({ \textstyle e-a)_n\,( f-a }) _{n}}  {({ \textstyle e)_n\,( f}) _{n}}}
  {} _{4} F _{3} \!\left [ \begin{matrix} { -n, a, 1 + a + c - e - f - n, 1 + a + b -
  e - f - n}
\nonumber
  \\
{ 1 + a + b + c - e - f - n, 1 + a - e - n, 1 + a - f -
    n}\end{matrix} ; {\displaystyle 1}\right ]   .
\end{align*}
where $n$ is a non-negative integer. As a result, we obtain
\begin{multline}
\frac {\M(\hat H_{a,b,c}^{i,j})} {\M(H_{a,b,c})}
=\frac {(a)_b\,(c)_{j-1}\,(1+b-j)_{i-1}\,(2+a-i-j)_{i+j-2}}
   {(a+c)_b\,(i-1)!\,(j-1)!\,(1+a+c-i)_{i-1}\,(1+a+b-j)_{j-1}}
   \\[10pt]
\times
\frac {(b)_{i-1}\,(1+a-i)_{i-1}}
{(1+b-j)_{i-1}\,(2+a-i-j)_{i-1}}
   {} _{4} F _{3} \!\left [ \begin{matrix} {1 - i, 1 - j, 1, 1 - a - b - c }
   \\
{ 2 - c - j, 2 - b - i, 1 - a }\end{matrix} ; {\displaystyle 1}\right ].
\label{eq:prod}
\end{multline}
Now we replace $a,b,c$ by~$n$, we replace $i,j$ by $n/2$, and then we
compute the asymptotics as $n\to\infty$. For the factorials and
Pochhammer symbols this is easily done using Stirling's formula.
For the $_4F_3$-series in the last line we may appeal to dominated
convergence and get
\begin{equation*} 
\lim_{n\to\infty}                                                               
{} _{4} F _{3} \!\left [ \begin{matrix} {1 - \frac {n} {2}, 1 - \frac {n} {2},         
1, 1 - 3n }\\                                                                   
{ 2 - \frac {3n} {2}, 2 - \frac {3n} {2}, 1 - n }\end{matrix} ; {\displaystyle 1}\
\right ]                                                                        
=\sum_{k=0}^\infty                                                              
\left(\frac {\left(-\frac {1} {2}\right)\cdot\left(-\frac {1} {2}\right)        
\cdot(-3)} {\left(-\frac {3} {2}\right)\cdot\left(-\frac {3} {2}\right)         
\cdot(-1)}\right)^k=\sum_{k=0}^\infty\left(\frac {1} {3}\right)^k               
=\frac {3} {2}.                                                                 
\end{equation*}
For the product in front of the $_4F_3$-series on the
right-hand side of~\eqref{eq:prod}, we obtain, after some simplification,
\begin{align*}
&
\frac {(c)_{j-1}\,(1+a-j)_{b}\,
(b)_{i-1}\,(1+a-i)_{i-1}}
  {(i-1)!\,(j-1)!\,(1+a+c-i)_{b+i-1}}
\nonumber
  \\[10pt]
&
=\frac {\Ga(c+j-1)\,\Ga(1+a+b-j)\,
\Ga(b+i-1)\,\Ga(a)\,\Ga(1+a+c-i)}
{\Ga(c)\,\Ga(1+a-j)\,\Ga(b)\,\Ga(1+a-i)\,
  \Ga(i)\,\Ga(j)\,\Ga(a+c+b)}
\nonumber
\\[10pt]
&
=\frac {\Ga^2(3n/2-1)\,\Ga^2(1+3n/2)}
{\Ga(n)\,\Ga^2(1+n/2)\,
  \Ga^2(n/2)\,\Ga(3n)}
\nonumber
\\[10pt]
&
=\left(\frac {27} {16}\right)^n
\frac {2} {3\sqrt{3}\pi n}\left(1+O\left(\frac {1} {n}\right)\right).
\end{align*}
Comparing this to what follows from Theorem~\ref{ted} when $A=B=C=1$ and $\al=\be=1/2$, one readily gets equation~\eqref{eex}. \end{proof}

\medskip\noindent
{\bf Acknowledgments.} We thank Michael Larsen, Russell Lyons and Bernhard Gittenberger for very helpful discussions concerning the scaling limit of a lattice path with fixed starting and ending points (see fact $(ii)$ in Section~2).

\newpage

\section*{Appendix: Scaling limit of lattice paths}
\setcounter{equation}{0}%
\global\def\theequation{\mbox{A.\arabic{equation}}}
\setcounter{theo}{0}%
\global\def\thetheo{\mbox{A.\arabic{theo}}}

\bigskip
\centerline{by Michael Larsen}
\smallskip
\centerline{\it Indiana University, Department of Mathematics, Bloomington, Indiana, USA}

\vskip 0.5in
Let $S=\{v_1,\ldots,v_n\}$ be a subset of $\Z^m$ such that there exists a  vector $u\in\R^n$ with $u\cdot v_i > 0$ for all $i$.  We define an \emph{$S$-path to $w\in \Z^m$} as a sequence $0=w_0, w_1,\ldots,w_N$ with $w_N=w$ and $w_i-w_{i-1} \in S$ for $1\le i\le N$.  (Our condition on $S$ guarantees that there are only finitely many $S$-paths with a given endpoint, i.e., that $N$ is bounded in terms of $w$.). An $S$-path \emph{deviates by $\epsilon>0$} if there exists $i\in [1,N]$ such that $|w_i - iw/N| > \epsilon N$.  

\begin{theo}
Given $S$ as above and $\epsilon>0$, there exists $c_1>0$  such that for all $w\in \Z^m$ with $|w|$ sufficiently large, the fraction of $S$-paths to $w$ that deviate by $\epsilon$ is less than $e^{-c_1|w|}$.
\end{theo}

\begin{proof}
We have
$$N \ge \frac{|w|}{\max_i |v_i|},$$
so it suffices to prove that, if $|w|$ is sufficiently large, then for every $N$ the fraction of $S$-paths of length $N$ to $w$
which deviate by $\epsilon$ is less than $e^{-c_2 N}$ for some fixed $c_2 > 0$ (by the mediant inequality: $a_1/A_1\leq b,\cdots,a_k/A_k\leq b$
implies $(a_1+\cdots+a_k)/(A_1+\cdots+A_k)\leq b$).
In turn, for this it suffices to prove that, if $|w|$ is sufficiently large, then for every $N$ and for each $M\in [1,N]$,
the fraction of $S$-paths to $w$ of
length $N$ such that
\begin{equation}
\label{M-deviation}
\left|w_M- \frac {Mw}N\right| > \epsilon N
\end{equation}
is less than $N^{-1}e^{-c_2 N}$ (because if a path deviates by $\epsilon$, then it deviates so for at least on ``time''~$M$).

Associated to each $S$-path of length $N$, we have an ordered $n$-tuple $(a_1,\ldots,a_n)$ of
non-negative integers with sum $N$, where $a_j$ is the number of steps of type $v_j$.  We call this
the \emph{type} of the $S$-path.  It suffices to prove that for each sufficiently large integer $N$,
each ordered $n$-tuple $(a_1,\ldots,a_n)$ 
of non-negative integers summing to $N$, and each $M\in [1,N]$, the fraction of all $S$-paths of type $(a_1,\ldots,a_n)$ 
satisfying \eqref{M-deviation} is less than $N^{-1}e^{-c_2 N}$ (use again the mediant inequality, with $k=\#(a_1,\dotsc,a_n)$'s).

Suppose $0=w_0,w_1,\ldots,w_N=w$ is an $S$-path of type $(a_1,\ldots,a_n)$ satisfying \eqref{M-deviation}.
Let $b_{i}$ denote the number of steps of type $v_i$ among the first $M$ steps.  Thus, $\sum_i b_{i} v_{i} = w_{M}$.
It follows that there exists a constant $c_3>0$ depending only on $S$ and $\epsilon$ such that for some $j\in [1,n]$, we have
\begin{equation}
\label{Mj-deviation}
\left|b_j - \frac {Ma_j}N\right| > c_3N
\end{equation}
(since
$\epsilon N<\left|w_M-\frac{M}{N}w\right|=
\left|\sum_{i=1}^n b_iv_i-\frac{M}{N}\sum_{i=1}^n a_iv_i\right|=
\left|\sum_{i=1}^n \left(b_i-\frac MN a_i\right)v_i\right|$).
  
To prove the theorem, using one more time the mediant inequality (with $k=\#(b_1,\dotsc,b_n)$'s)), we need only prove that, if $N$
is
large enough, then for each $M$ with $1\leq M\leq N$ and
all fixed $n$-tuples of non-negative integers $(a_1,\dotsc,a_n)$ with $\sum_i a_i =N$ and $(b_1,\dotsc,b_n)$ with $\sum_i b_i =M$, we have
$$\frac{M!}{b_1!\cdots b_n!}\frac{(N-M)!}{(a_1-b_1)!\cdots (a_n-b_n)!} \le \frac 1{N} e^{-c_2 N} \frac{N!}{a_1!\cdots a_n!}.$$

Now, \eqref{Mj-deviation} implies that there exists $k\in [1,n]$ such that
$$
\left|b_k - \frac {Ma_k}N\right| > N\frac{c_3}{n-1},
$$
and such that $b_k - \frac {Ma_k}N$ and $b_j - \frac {Ma_j}N$ have opposite signs
(since $\sum_{i=1}^n(b_i-Ma_i/N)=\break M-MN/N=0$).

Assuming first that $b_k > Ma_k/N$ and therefore $b_j < Ma_j/N$, we define 
$$L := \min\left(\left\lfloor \frac{Ma_j/N-b_j}2\right\rfloor,\left\lfloor \frac{b_k-Ma_k/N}2\right\rfloor\right),$$
so
$$b_j \leq \frac{M a_j}N - 2L\quad \text{and}\quad  b_k \geq \frac{M a_k}N + 2L.$$
We can bound $L$ below by an expression of the form $c_4N$ for
$c_4>0$,
where $c_4$ depends only on $S$ and $\epsilon$ (because $n$ depends only on $S$). 

For $1\le i\le L$, we have (note that $a_j-b_j\geq Ma_j/N-b_j\geq 2L$, so $a_j-b_j-L\geq L>0$)
\begin{equation*}
\begin{split}
\frac{b_j+i}{a_j-b_j-i} \leq \frac{b_j+L}{a_j-b_j-L}\le \frac{Ma_j/N - L}{(N-M)a_j/N+L} &\le \frac{Ma_j/N - c_4N}{(N-M)a_j/N+c_4N} \\
\le \frac{Ma_j/N - c_4M}{(N-M)a_j/N+c_4(N-M)}
&= \frac{a_j/N-c_4}{a_j/N+c_4} \frac M{N-M}. \\
\end{split}
\end{equation*}
Likewise,
\begin{equation*}
\begin{split}
\frac{a_k-b_k+i}{b_k-i} \leq \frac{a_k-b_k+L}{b_k-L}\le \frac{(N-M)a_k/N - L}{Ma_k/N+L} 
&\le\frac{(N-M)a_k/N - c_4(N-M)}{Ma_k/N+c_4M} \\
&= \frac{a_k/N-c_4}{a_k/N+c_4} \frac {N-M}M.\\
\end{split}
\end{equation*}
Multiplying both sides of these two inequalities, we obtain
\begin{equation}
\label{cancel}
\frac{b_j+i}{a_j-b_j-i} \frac{a_k-b_k+i}{b_k-i} < \frac{a_j/N-c_4}{a_j/N+c_4} \frac{a_k/N-c_4}{a_k/N+c_4} \leq \frac{1-c_4}{1+c_4} \frac{1-c_4}{1+c_4}< 1-c_5
\end{equation}
for $i=1,\dotsc,L$, where $c_5>0$ depends only on $S$ and $\epsilon$. 

We define $(b'_1,\ldots,b'_n)$ to be the vector $(b_1,\ldots,b_n)$ but with $b_j$ and $b_k$ replaced with
$b_j+L$ and $b_k-L$ respectively.  
Since
$$\frac{N!}{a_1!\cdots a_n!} \ge \frac{i!}{b'_1!\cdots b'_n!}\frac{(N-i)!}{(a_1-b'_1)!\cdots (a_n-b'_n)!},$$
it suffices to prove
$$\frac{i!}{b_1!\cdots b_n!}\frac{(N-i)!}{(a_1-b_1)!\cdots (a_n-b_n)!} \le \frac 1{N} e^{-c_2 N} 
 \frac{i!}{b'_1!\cdots b'_n!}\frac{(N-i)!}{(a_1-b'_1)!\cdots (a_n-b'_n)!}.$$
 For Case 1, this means
to show
\begin{equation}
\label{goal}
\frac{(b_j+1)(b_j+2)\cdots(b_j+L)}{(b_k-L+1)\cdots (b_k-1)b_k}\frac{(a_k-b_k+1)\cdots(a_k-b_k+L)}{(a_j-b_j-L+1)\cdots(a_j-b_j)} \le \frac 1{N} e^{-c_2 N}.
\end{equation}
The left hand side of \eqref{goal} can be written as
$$\frac{(b_j+L)(a_k-b_k+L)}{b_k(a_j-b_j)}\prod_{i=1}^{L-1} \frac{(b_j+i)(a_k-b_k+i)}{(b_k-i)(a_j-b_j-i)}.$$
By \eqref{cancel} and the inequalities $a_k, b_k, b_j\leq N$, $L\leq N$, $b_k\geq 2L$, $a_j-b_j\geq L$ and $L\geq c_4 N$,  this
is less than
$$\frac{2}{c_4^2} (1-c_5)^{L-1} \leq \frac{2}{c_4^2} (1-c_5)^{c_4N-1},$$
which, if $N$ is sufficiently large and $c_2$ is sufficiently small, implies \eqref{goal}.

If $b_k > ia_k/N$,  we define
$$L := \min\left(\left\lfloor \frac{b_j-Ma_j/N}2\right\rfloor,\left\lfloor \frac{Ma_k/N-b_k}2\right\rfloor\right),$$
and define $(b'_1,\ldots,b'_n)$ by replacing $b_j$ and $b_k$ with $b_j-L$ and $b_j+L$, respectively,
and the argument goes through as before.
\end{proof}

\end{document}